\documentclass{amsart}

\usepackage{amsmath,amssymb}
\usepackage{amsthm}
\usepackage{indentfirst}
\usepackage{graphicx}
\usepackage{hyperref}
\usepackage{xcolor}
\usepackage[margin=1.5in]{geometry}

\numberwithin{equation}{section}

\theoremstyle{plain}
\newtheorem{theorem}{Theorem}[section]
\newtheorem{lemma}[theorem]{Lemma}

\newtheorem{prop}[theorem]{Proposition}

\newtheorem{corollary}[theorem]{Corollary}

\theoremstyle{definition}
\newtheorem{defn}[theorem]{Definition}

\theoremstyle{remark}
\newtheorem*{remark}{Remark}

\theoremstyle{remarks}

\newcommand{\R}{\mathbb{R}}

\IfFileExists{mathabx.sty}%
  {\DeclareFontFamily{U}{mathx}{\hyphenchar\font45}%
   \DeclareFontShape{U}{mathx}{m}{n}{<->mathx10}{}%
   \DeclareSymbolFont{mathx}{U}{mathx}{m}{n}%
   \DeclareFontSubstitution{U}{mathx}{m}{n}%
   \DeclareMathAccent{\widebar}{0}{mathx}{"73}%
}{%
  \PackageWarning{mathabx}{%
    Package mathabx not available, therefore\MessageBreak substituting
    widebar with overline\MessageBreak }%
  \newcommand{\widebar}[1]{\overline{#1}}%
}

\newcommand{\la}{\lambda}
\newcommand{\e}{\varepsilon}

\newcommand{\Dom}{{\text{\rm Dom}}}
\newcommand{\K}{\mathcal{K}}

\newcommand{\Rn}{{\mathbb R}^n}
\newcommand{\supp}{{\text{\rm supp}}}

\AtBeginDocument{%
   \def\MR#1{}
}

\begin{document}

\title[Quasimode, eigenfunction and spectral projection bounds]{Quasimode, eigenfunction and spectral projection bounds  for Schr\"odinger operators on manifolds with  critically  singular potentials}

\thanks{M.D.B.~was partially supported by NSF Grant DMS-1565436, Y.S.~was partially supported by the Simons Foundation and
  C.D.S.~was supported in part by the NSF (NSF Grant DMS-1665373)
  and the Simons Foundation. }

\author{Matthew D. Blair}
\address[M.D.B.]{Department of Mathematics and Statistics,
University of New Mexico, Albuquerque, NM 87131,  USA}
\email{blair@math.unm.edu}

\author{Yannick Sire}
 \address[Y.S.]{Department of
  Mathematics, Johns Hopkins University, Baltimore, MD 21218, USA}
\email{sire@math.jhu.edu}

\author{Christopher D. Sogge} \address[C.D.S.]{Department of
  Mathematics, Johns Hopkins University, Baltimore, MD 21218, USA}
\email{sogge@jhu.edu}

\begin{abstract}We obtain quasimode, eigenfunction and spectral projection bounds for Schr\"odinger operators,
$H_V=-\Delta_g+V(x)$, on compact Riemannian manifolds $(M,g)$ of dimension $n\ge2$, which
extend the results of the third author~\cite{sogge88} corresponding to the case where $V\equiv 0$.  We are
able to handle critically singular potentials and consequently assume that $V\in L^{\tfrac{n}2}(M)$
and/or $V\in {\mathcal K}(M)$ (the Kato class).
Our techniques involve combining arguments for proving quasimode/resolvent estimates for  the case where $V\equiv 0$ that go back to the third author \cite{sogge88} as well as ones which arose in the work of Kenig, Ruiz and this author~\cite{KRS} in the study of ``uniform Sobolev estimates''  in $\Rn$.  We also use techniques from more recent developments of several authors concerning variations on the latter theme in the setting of compact manifolds.
Using the spectral projection bounds  we can prove a number of
natural $L^p\to L^p$ spectral multiplier theorems under the assumption that $V\in L^{\frac{n}2}(M)\cap {\mathcal K}(M)$.
Moreover, we can also obtain natural analogs of the original Strichartz estimates~\cite{Strichartz77} for solutions of
$(\partial_t^2-\Delta +V)u=0$.  We also are able to obtain analogous results in $\Rn$ and state some global problems
that seem related to works on absence of embedded eigenvalues for
Schr\"odinger operators in $\Rn$ (e.g., \cite{IonescuJerison}, \cite{JK}, \cite{KenigNar},
\cite{KochTaEV} and \cite{iRodS}.)
\end{abstract}

\maketitle

\section{Introduction and main results}

The purpose of this paper is to obtain quasimode, eigenfunction and spectral projection bounds for Schr\"odinger operators,
\begin{equation}\label{a.1}
H_V=-\Delta_g+V(x),\end{equation}
on compact Riemannian manifolds $(M,g)$ of dimension $n\ge 2$.  We shall deal with real valued potentials
$V(x)$ with critical singularities.   Consequently, we shall assume throughout that $V$ is real valued and
\begin{equation}\label{a.2}
V\in L^{\frac{n}2}(M).
\end{equation}
Note that, in $\Rn$,  multiplication by elements of $L^{\frac{n}2}(\Rn)$ scales as operating by the Euclidean Laplacian does.
For most of our results we shall also have to assume that $V$ belongs to the Kato class, ${\mathcal K}(M)$, which
will be recalled in Definition~\ref{kato} below.

If we merely assume that \eqref{a.2} is valid the operator $H_V$ need not be self-adjoint.  Notwithstanding, in higher dimensions,
we can prove the following.

\begin{theorem}\label{theorem1}  Assume that $n\ge 4$ and $V\in L^{\frac{n}2}(M)$ and let
\begin{equation}\label{a.3}
\sigma(p)=\min\bigl( \, n(\tfrac12-\tfrac1p)-\tfrac12, \, \tfrac{n-1}2(\tfrac12-\tfrac1p)\, \bigr).
\end{equation}
Then for $\la\ge1$ we have
\begin{multline}\label{a.4}
\|u\|_{L^p(M)}\le C_{p,V}\bigl(\, \la^{\sigma(p)-1}\bigl\| \bigr(-\Delta_g+V-(\la+i)^2\bigl)u\bigr\|_{L^2(M)}
+\la^{\sigma(p)}\|u\|_{L^2(M)} \, \bigr),
\\ \text{if } \, \, \, u\in C^\infty(M),
\end{multline}
provided that
\begin{equation}\label{a.5}
2<p<\tfrac{2n}{n-3}.
\end{equation}
The constant $C_{p,V}$ depends on $p$, $V$ and $(M,g)$ but not on $\la$.
\end{theorem}

If $H_V$ were self-adjoint and positive and if $C^\infty(M)$ were an operator
core for $H_V$ and if we set $P_V=\sqrt{H_V}$, then \eqref{a.4} would yield the
spectral projection bounds
$$\|\chi^V_\la f\|_p\lesssim (1+\la)^{\sigma(p)}\|f\|_2, \quad \la\ge 0,$$
for $p$ as in \eqref{a.5}, where $\chi^V_\la$ is the spectral projection operator for $P_V$ corresponding to the unit interval $[\la,\la+1]$.

In the case where $V\equiv0$ in \cite{sogge88} the third author proved \eqref{a.4} for all $n\ge2$ with
\begin{equation}\label{a.6}
p=p_c=\tfrac{2(n+1)}{n-1}.
\end{equation}
This special case where $p=p_c$ yields the bounds in
\eqref{a.4} for $2<p<p_c$ by H\"older's inequality since the case where $p=2$ is trivial.  Heat kernel techniques
also imply that when $V\equiv 0$ the special case where $p=p_c$ yields the bounds for all $p\ge p_c$ if $n=2$ or $3$
as well as the bounds for $p_c<p<\infty$ when $n=4$ and $p_c<p\le \tfrac{2n}{n-4}$ if $n\ge 5$, see \S \ref{heatsection}.
In order to use these techniques to extend the ``quasimode'' bounds in Theorem~\ref{theorem1} for $H_V$ to such exponents, we shall have to assume
that, in addition to \eqref{a.2}, $V$ belongs to the Kato class that we shall define in a moment.  We shall also be able
to handle $p=\infty$ for $n=4$ and $p\ge \tfrac{2n}{n-4}$ for $n\ge5$ using heat equation techniques if we include an additional
term in the right to account for the unfavorable Sobolev embeddings for such exponents.

Koch, Tataru and Zworski~\cite{KTZ} also obtained semiclassical variants of \eqref{a.4} for all dimensions and all
exponents $2<p\le \infty$ under the assumption that $u$ is spectrally localized.  This assumption is needed since, as
we shall see, when $V\equiv 0$ \eqref{a.4} does not hold for $p=\infty$ if $n=4$ or $p>\tfrac{2n}{n-4}$ if $n\ge 5$.

The proof of quasimode estimates like \eqref{a.4} involves combining the resolvent/oscillatory integral
approach of the third author in \cite{sogge88} with techniques of Kenig, Ruiz and the third author~\cite{KRS} that were
used to prove ``uniform Sobolev inequalities'' in $\Rn$, $n\ge 3$:
\begin{multline}\label{a.7}
\|u\|_{L^s(\Rn)}\le C_{r,s}\|(\Delta+z)u\|_{L^r(\Rn)}, \quad z\in {\mathbb C},
\\
\text{if } \, \, u\in {\mathcal S}(\Rn), \, \, \, n(\tfrac1r-\tfrac1s)=2 \, \, \, \text{and } \, \, \,
s\in (\tfrac{2n}{n-1}, \tfrac{2n}{n-3}).
\end{multline}
As was shown in \cite{KRS}, the condition on the exponents is necessary.  The last condition
 accounts for the limitation in \eqref{a.5}.  On the other hand if,
in addition to \eqref{a.2}, we also assume that $V$ belongs to the Kato class, we can obtain \eqref{a.4} for the larger
(and essentially sharp) range $2<p<\tfrac{2n}{n-4}$ when $n\ge4$.

In addition to borrowing from the techniques of \cite{KRS}, we shall rely on arguments used more recently to prove
variants of \eqref{a.7} for compact manifolds.  Of course \eqref{a.7} cannot hold for all $z\in {\mathbb C}$ since
the right hand may be zero if $z$ is in the spectrum of $-\Delta_g$.  An appropriate variant for compact
manifolds reads as follows
\begin{equation}\label{a.8}
\|u\|_{L^{p'}(M)} \le C\bigl\| (-\Delta_g-(\la+i)^2)u\|_{L^p(M)}, \quad
\text{if } \, \, n(\tfrac1{p'}-\tfrac1p)=2,
\end{equation}
where, as usual $p'$ denotes the conjugate exponent for $p$ (i.e., $\tfrac1p+\tfrac1{p'}=1$).
These estimates were proved by Dos Santos Ferreira, Kenig and Salo~\cite{DKS}.  In the work of Bourgain, Shao and Yao and
the third author \cite{BSSY} it was shown that \eqref{a.8} is sharp in the sense that when $M=S^n$ one cannot
have the variant of the inequality where $(\la+i)^2$ is replaced by $(\la+\e(\la)i)^2$ with $\e(\la)\searrow 0$; however, it was
also shown that certain improvements of this type are possible under certain curvature assumptions.

One of course sees similarities between \eqref{a.4} and \eqref{a.8} since both involve
the parameter $(\la+i)^2$.  In \cite{HuangSogge}, Huang and the third author also showed that when $M=S^n$
the variant of \eqref{a.8} holds involving the exponents in \eqref{a.7}.  The proof of our quasimode estimates
will rely on techniques from \cite{HuangSogge} as well as the earlier works \cite{DKS} and \cite{ShaoYao}, all of which
allow one to show that the ``local operators'' that arise have desirable bounds for $p$ as in \eqref{a.5}.

The proof of Theorem~\ref{theorem1} also shows that the inequality holds when $n=3$ and $2<p<\infty$
if $V\in L^{\frac32}(M)$.  This, however, is not a useful inequality due to the fact that $L^{\frac{n}2}$ is not
contained in $L^2$ when $n=2,3$ and so the right side of \eqref{a.4} may be infinite for typical $u\in C^\infty(M)$.
We shall get around this nuisance by proving that if $V$ also belongs to the Kato class we have the variant of
\eqref{a.4} where $u$ ranges over the domain of $H_V$, $\Dom(H_V)$, and $2<p\le \infty$ for $n=3$.  We shall also
be able to prove quasimode bounds for these exponents when $n=2$.

Before stating these, let us go over the definition of the Kato class, ${\mathcal K}(M)$.  To do this,
let for $r>0$
\begin{equation}\label{a.9}
h_n(r)=
\begin{cases}
|\log r|, \quad \text{if } \, n=2
\\
r^{2-n},\quad \text{if } \, n\ge3.
\end{cases}
\end{equation}

\begin{defn}\label{kato}  The potential
$V$ is said to be in the Kato class
 and written as
$V\in {\mathcal K}(M)$ if
\begin{equation}\label{a.10}\lim_{r\searrow 0} \sup_{x}\int_{B_r(x)}
h_n\bigl(d_g(x,y)\bigr) \, |V(y)| \, dy=0,\end{equation}
where $d_g(\, \cdot \, , \, \cdot \, )$ denotes geodesic distance and $B_r(x)$ is the
geodesic ball of radius $r$ about $x$ and $dy$ denotes the volume element on $(M,g)$.
\end{defn}

Note that since $M$ is compact we automatically have that
$V\in L^1(M)$ if $V\in {\mathcal K}(M)$.
An easy argument also shows that if $V\in L^{\frac{n}2+\e}(M)$, $\e>0$, then $V\in {\mathcal K}(M)$; however,
$L^{\frac{n}2}(M)$ is not contained in ${\mathcal K}(M)$ or vice versa.
Moreover, as we shall review in the next section, if $V\in {\mathcal K}(M)$, then $H_V$
(defined as a sum of quadratic forms) is self-adjoint and bounded from below.
After adding a constant to the potential we may, and always shall assume that $H_V$ is positive when $V\in {\mathcal K}(M)$, cf. \S\ref{katosection}.


For more background on the Kato
class and related spaces we refer the reader to Simon~\cite{SimonSurvey} which deals with Schr\"odinger operators on $\Rn$; however,
most of the results there carry over without difficulty to our setting.

Let us now state our other main result.

\begin{theorem}\label{theorem2}  Assume that $V\in L^{\frac{n}2}(M)\cap {\mathcal K}(M)$.  Then
if $n=2$ or $n=3$,  $\sigma(p)$ as in \eqref{a.3} and $\la\ge 1$ we have
\begin{multline}\label{a.11}
\|u\|_{L^p(M)}\le C_V \la^{\sigma(p)-1}\bigl\| \, \bigl(-\Delta_g+V-(\la+i)^2\bigr)u\, \bigr\|_{L^2(M)},
\\
\text{if } \, \, 2<p\le \infty \quad \text{and } \, \, u\in \Dom(H_V).
\end{multline}
If $n\ge 4$ this inequality holds for all $2<p<\tfrac{2n}{n-4}$, and we also have for such $n$
\begin{multline}\label{a.12}
\|u\|_{L^p(M)}\le C_V \Bigl( \la^{\sigma(p)-1}\bigl\|  \bigl(-\Delta_g+V-(\la+i)^2\bigr)u \bigr\|_{L^2(M)}
\\
+\la^{-N+n/2}\bigl\|(I+H_V)^{N/2}R_\la u\bigr\|_{L^2(M)}\Bigr), \\ \text{if } \, \,
p\in [\tfrac{2n}{n-4}, \infty], \, \,  \text{and } \, \, u\in \Dom(H_V), \quad \la\ge1,
\end{multline}
assuming that $N>n/2$ with $R_\la$ being the projection operator for $P_V=\sqrt{H_V}$ corresponding
to the interval $[2\la,\infty)$.
\end{theorem}

We could have stated \eqref{a.11} as in \eqref{a.4} with the additional term in the right; however,
since $H_V$ is self-adjoint under the above assumptions, by the spectral theorem, this term is redundant by which we mean that  \eqref{a.11} or the variant including $\la^{\sigma(p)}\|u\|_2$ in the right are equivalent.

If $\chi^V_\la$ is the spectral projection operator associated with $P_V$ corresponding to the unit intervals $[\la,\la+1]$,
then we have the following corollary.

\begin{corollary}\label{spec}
Let $n\ge 2$ and $\sigma(p)$ be as in \eqref{a.3}.  Then if $V\in L^{\frac{n}2}(M)\cap {\mathcal K}(M)$
\begin{equation}\label{a.13}
\|\chi_\la^V f\|_{L^p(M)}\le C_V (1+\la)^{\sigma(p)}\|f\|_{L^p(M)}, \quad p\ge 2, \, \, \la\ge 0.
\end{equation}
Consequently, if
\begin{equation}\label{a.14}(-\Delta_g+V)e_\la = \la^2 e_\la
\end{equation}
in the sense of distributions, we have
\begin{equation}\label{a.15}
\|e_\la\|_{L^p(M)}\le C_V(1+\la)^{\sigma(p)}\|e_\la\|_{L^2(M)}, \quad p\ge 2, \, \, \la\ge 0.
\end{equation}
\end{corollary}

It is well known that for all exponents $p>2$ when $n=2$ and $n=3$ and for relatively small exponents  (including the ``critical'' exponent $p=p_c$)
in dimensions $n\ge4$, bounds
of the form \eqref{a.13} imply quasimode estimates of the form \eqref{a.4} or \eqref{a.11}.  See  \cite[Theorem 1.1]{SoggeZelditchQMNote}.

Based on the work of the third author \cite{sogge87}, \cite{sogge02}, and the third author with Seeger \cite{SeegerSoggePDOfunction}, it is known that in many cases one can use spectral projection bounds to prove multiplier theorems.  In \S\ref{specmult}, we shall show that \eqref{a.13} implies sharp bounds for the Bochner-Riesz operators associated with $H_V$ when $p\in [1,p_c']\cup [p_c,\infty]$ with $p_c$ as in \eqref{a.6}.  Moreover, we observe consequences for spectral multiplier theorems of H\"ormander-Mikhlin type.

We would now like to remark that \eqref{a.15} need not hold for $p=\infty$ if we drop the assumption that $V\in {\mathcal K}(M)$. To see this, we shall take $M$ to be the round sphere $S^n$ with $n\ge 2$.  We then can write
$$S^n = \bigl\{\,  (\omega \sin\phi, \cos \phi): \, \phi\in [0,\pi], \, \, \omega\in S^{n-1}\, \bigr\}.$$
Then if $f$ is a function on $S^n$ only depending on $\phi$ (i.e., distance from the poles $(\pm 1,0,\dots,0)$), one has
$$\Delta_{S^n}f =(\sin \phi)^{-(n-1)} \frac{\partial}{\partial \phi}\, \Bigl( \bigl(\sin\phi\bigr)^{n-1} \,
\frac{\partial f}{\partial \phi}\Bigr).$$

Let us first handle the case where $n\ge 3$.  We let
$$f=-\ln\bigl(\tfrac12 \sin\phi\bigr)$$
so that $f\ge \ln 2>0$ on $S^n$.  Then,
\begin{align*}
-\Delta_{S^n}f&=(\sin\phi)^{-(n-1)}\frac{\partial}{\partial \phi}\bigl(\, (\sin\phi)^{n-2} \cos\phi\, \bigr)
\\
&=\bigl(\sin\phi\bigr)^{-2}\bigl((n-2)\cos^2\phi - \sin^2\phi\bigr).
\end{align*}
Thus, if
$$V=\frac{(n-2)\cos^2\phi-\sin^2\phi}{\sin^2\phi \, \ln(\tfrac12 \sin\phi)},$$
we have that $V\approx \frac{-(n-2)}{\sin^2\phi |\ln \sin\phi|}\ll 0$ near $\phi=0$ and $\phi=\pi$,
and
$$H_V f =0\cdot f,$$
so that $f=e_0$ is an unbounded eigenfunction with eigenvalue $0$, which means that \eqref{a.15} cannot
hold in this case when $\la=0$.

Note that, as $n\ge3$,
$$V\in L^{\frac{n}2}(S^n),$$
but
$$V\notin {\mathcal K}(S^n) \cup L^{\frac{n}2+\delta}(S^n), \quad \text{if } \, \, \delta>0.$$
 Thus, we conclude that merely assuming $V\in L^{\frac{n}2}(M)$ is not sufficient to get $O((1+|\la|)^{\frac{n-1}2})$ sup-norm estimates for eigenfunctions of Schr\"odinger operators on compact manifolds when $n\ge 3$.

To handle the case where $n=2$ one needs to modify this argument.  Here, we take
$f=[\ln(\tfrac12 \sin\phi)]^2$.  Then
\begin{align*}
-\Delta_{S^2}f&=-2(\sin \phi)^{-1}\, \frac{\partial}{\partial \phi}\bigl( \cos \phi \cdot \ln(\tfrac12 \sin\phi)\bigr)
\\
&=2\ln(\tfrac12\sin\phi)-2(\sin\phi)^{-2}\cos^2\phi,
\end{align*}
and so
if
$$V=2\frac{\cos^2\phi-\sin^2\phi \cdot \ln(\tfrac12\sin \phi)}{\sin^2\phi \cdot (\ln\tfrac12 \sin\phi)^2},$$
then
$$H_Vf=0\cdot f.$$
Like before,
$$V\notin {\mathcal K}(S^2) \cup L^{1+\delta}(S^2), \, \, \text{if } \, \delta>0,
 \quad \text{but } \, V\in L^1(S^2).$$
 Since $f=e_0$ is an unbounded eigenfunction with eigenvalue $0$, we conclude that
 \eqref{a.15} also breaks down on $S^2$ if we do not assume that $V\in {\mathcal K}(S^2)$.

Comparing Theorems \ref{theorem1} and \ref{theorem2} shows that the assumption $V\in {\mathcal K}(M)$ only
enters when proving quasimode estimates for large exponents $p$.  Indeed, by \eqref{a.4}
we have \eqref{a.15} for all $2<p<\tfrac{2n}{n-3}$ if $n\ge 4$. The example we have just given
 only shows that the
bounds need not hold for $p=\infty$.  It would be interesting\footnote{Simon~\cite[\S A.3]{SimonSurvey} raises an analogous problem for $L^p$ bounds for eigenfunctions of Schr\"odinger operators in $\Rn$ but says that the ``class of potentials .... includes none of physical interest".  This is due to the fact that the associated operators $H_V$ need not be  essentially self-adjoint if one weakens the  hypotheses
in Corollary~\ref{spec}.} to determine if we might have \eqref{a.15} for a larger
range than $2<p<\tfrac{2n}{n-3}$ when $n\ge4$.  Results of Brezis and Kato~\cite[Theorem 2.3]{BeKa} for $\Rn$ suggest
that, like for the above counterexample, only the case where $p=\infty$ may violate \eqref{a.15}.

The paper is organized as follows.  In the next section we shall go over background concerning the Kato
class and also review the facts about the Hadamard parametrix and the oscillatory integral bounds that we shall use
in proving our quasimode estimates.  Then in \S3 we shall prove Theorem~\ref{theorem1}.  In \S \ref{3dsection}
and \ref{2dsection} we shall prove
the bounds in Theorem~\ref{theorem2} for $n=3$ and $n=2$.  Different arguments
are needed for these two cases due to the nature of Sobolev embeddings and
the fact that the uniform Sobolev estimates in \cite{KRS} (and manifold variants)  do not hold when $n=2$.  In \S5 we shall prove the remaining part
of this theorem corresponding to $n\ge4$.
In \S \ref{specmult}--\ref{strichartzsection} we shall go over applications, showing that we can use the spectral
projection estimates to prove natural multiplier theorems, and, moreover, Strichartz estimates for wave
operators involving potentials $V\in L^{n/2}(M)\cap {\mathcal K}(M)$.
In the final section we shall show how our results extend to Schr\"odinger operators $H_V$ in $\Rn$
and go over some natural global problems, such
obtaining
improved spectral projection estimates and the related problem of
proving global Strichartz estimates, that remain open and seem to be related to work on
proving that embedded eigenvalues do not exist (e.g., \cite{IonescuJerison}, \cite{JK}, \cite{KenigNar},
\cite{KochTaEV} and \cite{iRodS}).

\section{Some background}\label{background}

In this section we shall collect the main facts that we shall require.  We shall review how the assumption
$V\in {\mathcal K}(M)$ implies that the symmetric operators $H_V$ in \eqref{a.1} are self-adjoint and bounded
from below.  We shall also review facts about the Hadamard parametrix and bounds for  the oscillatory integral operators
that will arise in our proofs.

Let us start out with the former.

\subsection{The Kato class and self-adjointness}\label{katosection}

As we stated before, for brevity, here and throughout, $dx$ shall denote the Riemannian measure on $(M,g)$.

\begin{prop}\label{self-adjoint}
If $V\in {\mathcal K}(M)$ the quadratic form,
$$q_V(u,v)=-\int_M Vu \, \overline{v}\, dx +\int -\Delta_g u \, \overline{v}\, dx, \quad
u,v \in \Dom(\sqrt{-\Delta_g+1}),$$
is bounded from below and defines a unique semi-bounded self-adjoint operator $H_V$ on $L^2$.  Moreover, $C^\infty(M)$ constitutes a form core\footnote{Recall that a \emph{form core} for $q_V$ is a subspace $S$ which approximates elements $u$ in the domain of the form in that there exists a sequence $u_m \in S$ satisfying $\lim_m \|u-u_m\|^2 + q_V(u-u_m,u-u_m) = 0$.} for $q_V$.
\end{prop}

\begin{proof}
Since $(-\Delta_g+1)^{1/2}$ is self-adjoint, by perturbation theory (specifically the KLMN Theorem
(see \cite[Theorem X.17]{ReedSimon}) it suffices to prove that for any $0<\e<1$ there is a constant
$C_\e<\infty$ so that
\begin{equation}\label{s.1}
\int |V| \, |u|^2 \, dx \le \e^2 \, \bigl\|(-\Delta_g+1)^{1/2}u\bigr\|_2^2 +C_\e\|u\|_2^2, \quad
u\in \Dom(\sqrt{H_0}),
\end{equation}
where $H_0=-\Delta_g+1$.

To prove this, following the argument, for instance, in \cite[Proposition A.2.3]{SimonSurvey} for $\Rn$,  we shall use the fact that the heat kernel $p_t(x,y)=\bigl(e^{-tH_0}\bigr)(x,y)$ for $H_0$ satisfies
\begin{equation}\label{s.2}
0\le p_t(x,y)\le
\begin{cases}
C_0t^{-n/2} \, \exp(-c_0(d_g(x,y))^2/t), \quad 0\le t\le 1
\\
\exp(C_0t), \quad t\ge 1,
\end{cases}
\end{equation}
where $c_0>0$ and $C_0<\infty$ are uniform constants.  These are a consequence of the Li-Yau estimates in \cite{LiYau}.  Using this and the
definition
\eqref{a.10} of ${\mathcal K}(M)$, we see\footnote{To see this we note that by \eqref{s.2}, if $N$ is large enough the $t$-integral is dominated by
$h_n(d_g(x,y))$.  Thus, as $V\in {\mathcal K}(M)$, we just need to see that if the $y$-integral is taken over the region where $\{y\in M: \, d_g(x,y)>\delta\}$, with $\delta>0$ fixed, then the resulting expression is small.  Since this also follows easily from \eqref{s.2} our claim follows.}  that if $V\in {\mathcal K}(M)$
$$\sup_{x\in M} \int_0^\infty \int_M e^{-Nt}\,
p_t(x,y) \, |V(y)| \, dy dt \to 0, \quad \text{as } \, \, N\to \infty.$$
Choose $N=N_\e$ so that the left side is $<\e^2$, i.e.,
$$\bigl\| \,  (H_0+N_\e)^{-1}|V|\, \bigr\|_\infty<\e^2.$$

This means that the operator $u\to (H_0+N_\e)^{-1}(|V|u)$ satisfies
$$\bigl\|(H_0+N_\e)^{-1}|V| \bigr\|_{L^\infty\to L^\infty}<\e^2.$$
By duality, we also have
$$\bigr\||V|(H_0+N_\e)^{-1}\bigr\|_{L^1\to L^1}<\e^2.$$
An application of Stein's interpolation theorem therefore yields
\begin{equation*}
\bigl\||V|^{1/2}(H_0+N_\e)^{-1}|V|^{1/2}\bigr\|_{L^2\to L^2}<\e^2,
\end{equation*}
which, by a $TT^*$ argument, is equivalent to
\begin{equation}\label{semibounded}
\bigl\| |V|^{1/2}(H_0+N_\e)^{-1/2}\bigr\|_{L^2\to L^2}<\e.
\end{equation}
Since this implies \eqref{s.1} with $C_\e=N_\e$, we are done.
\end{proof}

If $u\in \Dom(\sqrt{-\Delta_g+1})$ then $-\Delta_gu$ and $Vu$ are both distributions.  If $H_V$ is the self-adjoint
operator given by the proposition, then $\Dom(H_V)$ is all such $u$ for which $-\Delta_gu+Vu\in L^2$.  At times,
such as in the statement of Theorem~\ref{theorem2} we abuse notation a bit by writing $H_V$ as $-\Delta_g+V$.

Note that \eqref{semibounded} implies that $q_V$ is bounded from below.  If we take $\e^2=1/2$ in \eqref{s.1}
we indeed get for large enough $N$
\begin{multline}\label{semi}
\|\sqrt{-\Delta_g+1}u\|_2^2 =\int (-\Delta_g+1)u \, \overline{u} \, dy
\le 2\int(-\Delta_g+V+N)u\, \overline{u}\, dy
\\
=2\bigl\|\sqrt{H_V+N}u\bigr\|_2^2, \quad \text{if } H_V=-\Delta_g+V.
\end{multline}
Thus, $(-\Delta_g+1)^{1/2}(H_V+N)^{-1/2}$ and $(H_V+N)^{-1/2}(-\Delta_g+1)^{1/2}$ are bounded on $L^2$.
Since $(-\Delta_g+1)^{-1/2}$ is a compact operator on $L^2$, so must be $(H_V+N)^{-1/2}$.
From this we conclude that the self-adjoint operator $H_V$ has {\em discrete spectrum}.  By heat kernel methods one can also show that
the eigenfunctions $e_\la$ of $(H_V+N)$ are continuous.  (See, e.g., \cite[Theorem 2.21]{Guneysu} and \cite{Sturm}).

After replacing $V$ by $V+N$ to simplify the notation, we may assume, as we shall in what follows, that
\eqref{semi} holds with $N=0$.   This just shifts the spectrum and does not change the eigenfunctions.  In this case the spectrum of $H_V$ is positive and its eigenfunctions therefore
are distributional solutions of
$$H_V e_\la =\la^2 e_\la, \quad \text{some }\, \la>0,$$
which means here that $\la$ is the eigenvalue of the ``first order'' operator $\sqrt{H_V}$, i.e.,
\begin{equation}\label{ef}
P_V e_\la =\la e_\la, \quad \text{if } \, \, P_V=\sqrt{H_V}.
\end{equation}

\subsection{The Hadamard parametrix and oscillatory integral bounds}\label{hadamardsection}

As in many early works (e.g., \cite{DKS}, \cite{HuangSogge} and \cite{sogge88}), we shall prove
our estimates using the Hadamard parametrix.  Let us quickly review the facts that we shall require.
More details can be found in these works as well as in \cite[\S 17.4]{HIII} and \cite[\S 2.4]{SoggeHangzhou}.

Recall that we are abusing the notation a bit by letting $dx$ denote the volume element associated with the metric $g$ on $M$ and all integrals are to be
taken with this measure.  In local coordinates it is of the form $|g|^{1/2}$ times Lebesgue measure,
where $|g|=\det(g_{jk}(x))$.
Here $g_{jk}(x)dx^jdx^k$ is the metric.  In local coordinates the Laplace-Beltrami operator $\Delta_g$ takes the form
$$|g|^{-\frac12}\sum_{j,k=1}^n \frac\partial{\partial x_j}\bigl( |g|^{\frac12} g^{jk}(x)\bigr)\frac\partial{\partial x_k},$$
and so $\Delta_g$ is self-adjoint with respect to the volume element.  Also for $x$ sufficiently close to $y$ we shall let
$d_g(x,y)$ denote the geodesic distance between $x$ and $y$.

The Hadamard parametrix for
$-\Delta_g-(\la+i)^2$ is an approximate
``local inverse'' that is built using the radial functions
\begin{equation}\label{h.1}
F_\nu(|x|,\la)=\nu !  \, (2\pi)^{-n}\int_{\Rn}\frac{e^{ix\cdot \xi}}{\bigl(|\xi|^2-(\la+i)^2\bigr)^{\nu+1}} \, d\xi,
\quad \nu=0,1,2,3,\dots.
\end{equation}
Here $|x|$ denotes the Euclidean length of $x\in \Rn$.  If $\Delta=\Delta_{\Rn}$ denotes the Euclidean Laplacian, then
$$\bigl(-\Delta-(\la+i)^2\bigr) F_0(|x|,\la)=\delta_0,$$
whilst
$$\bigl(-\Delta-(\la+i)^2\bigr) F_\nu =\nu F_{\nu-1}, \quad \text{if } \, \, \, \nu=1,2,3,\dots.$$
Here and throughout we are always assuming that $\la\ge1$.

Using these equations one can find coefficients $\alpha_\nu \in C^\infty$ defined near the diagonal so that for $N\in
{\mathbb N}$ we have for $x$ near $y$
\begin{equation}\label{h.2}
\bigl(-\Delta_g-(\la+i)^2\bigr) F =\bigl(\det g_{jk}(x)\bigr)^{\frac12} \, \delta_y(x)-(\Delta_g \alpha_N)F_N,
\end{equation}
if
\begin{equation}\label{h.3}
F(x,y,\la)=\sum_{\nu=0}^N \alpha_\nu(x,y)F_\nu(d_g(x,y),\la)
\end{equation}
and of course $F_N=F_N(d_g(x,y),\la)$ in \eqref{h.2}.
By choosing $N$ large enough (depending on the dimension) we can ensure that the last term
is bounded, i.e.,
\begin{equation}\label{h.4}
|(\Delta_g\alpha_N)F_N|\le C_0 \quad \text{if } \, \, \, \la\ge 1.
\end{equation}
The identity \eqref{h.2} is just (17.4.6)$'$ in \cite{HIII}.

We shall also need more information about the functions $F_\nu$ in \eqref{h.1}.  Specifically, we recall that we can
rewrite them as
\begin{equation}\label{h.5}
F_\nu(r,\la)=c_\nu r^{-\frac{n}2+\nu+1} \, z^{\frac{n}4-\frac{\nu+1}2} \,
K_{\frac{n}2-\nu-1}(\sqrt{z}r), \quad
z=-(\la+i)^2,
\end{equation}
where $K_m$ are the modified Bessel functions of the second kind defined by
$$K_m(z)=\int_0^\infty e^{-z\cosh t}\cosh(mt) \, dt, \quad \text{Re }z>0.$$
As is well known (see \cite{formulas})
\begin{multline}\label{h.6}
|K_m(z)|\le C_m |z|^{-m} \quad \text{if } \, \, m>0, \, \, \quad \text{while } \, \,
|K_0(z)| \le C |\log(|z|/2)|,
\\
\quad \text{when } \, \, |z|\le 1 \quad \text{and } \, \, \text{Re }z>0,
\end{multline}
and also
\begin{equation}\label{h.7}
K_m(z)=a_m(z) \, z^{-\frac12} e^{-z} \quad \text{when } \, \, |z|\ge 1 \quad \text{and } \, \, \text{Re }z>0,
\end{equation}
where the $a_m$ behave like symbols, i.e.,
\begin{equation}\label{h.8}
\Bigl| \frac{d^j}{dr^j} a_m\bigl(r \tfrac{z}{|z|}\bigr)\Bigr| \le C_{j,m} r^{-j}, \quad
j=0, 1, 2, \dots \quad \text{if } \, \, \, \text{Re }z>0 \quad \text{and } \, \, r\ge 1.
\end{equation}
More details can be found in \cite[p. 338--339]{KRS}, \cite[Lemma 4.3]{sogge88} and \cite{DKS}.

From \eqref{h.5}---\eqref{h.8} we deduce the following result which is essentially Lemma 4.3 in \cite{sogge88}.

\begin{lemma}\label{lemmah1}  There is an absolute constant $C$ so that for $\la\ge1$
$$|F|\le C(d_g(x,y))^{2-n} \quad \text{if } \, d_g(x,y)\le \la^{-1} \quad \text{and } \, \, n\ge 3,$$
and
$$|F|\le C|\log (\la d_g(x,y)/2)| \quad  \text{if } \, d_g(x,y)\le \la^{-1} \quad \text{and } \, \, n=2.$$
Furthermore, for $d_g(x,y)$ smaller than a fixed constant (depending on $(M,g)$)
$$F=\la^{\frac{n-3}2} e^{-i\la d_g(x,y)}\, \bigl(d_g(x,y)\bigr)^{-\frac{n-1}2} \, a_\la(x,y),
\quad \text{if } \, \, d_g(x,y)\ge \la^{-1},
$$
where
$$|\nabla_{x,y}^\alpha a_\la(x,y)|\le C_\alpha \bigl(d_g(x,y)\bigr)^{-|\alpha|}.
$$
\end{lemma}

As we pointed out before, \eqref{h.2} is only valid near the diagonal, as  is the representation of $F$ as in the
last part of this lemma.  Due to this, as well as to be able to exploit our assumptions regarding the potentials, let
us introduce cutoffs.  Specifically, fix $\eta \in C^\infty_0([0,\infty))$ which equals one for $s\le 1/2$ and zero
for $s\ge 1$ and set
$$\eta_\delta(x,y)=\eta\bigl(d_g(x,y)/\delta\bigr).$$
Of course this cutoff then satisfies the bounds for $a_\la$ above.

Next, if $\delta>0$ is sufficiently small, by \eqref{h.2}, we have
\begin{multline}\label{h.2'}\tag{2.7$'$}
\bigl(-\Delta_g-(\la+i)^2\bigr) \Bigl(\eta_\delta(x,y)F(\, \cdot \, ,y,\la)\Bigr)
=\bigl(\det g_{jk}(x)\bigr)^{\frac12} \, \delta_y(x)
\\
-\eta_\delta(x,y) \, \Delta_g\alpha_N \cdot F_N(\, \cdot \, ,y,\la)
+[\eta_\delta(\, \cdot \, ,y), \Delta_g] \, F(\, \cdot \, ,y,\la).
\end{multline}
We think of the last two terms as ``remainder terms''.  As we pointed out before, by \eqref{h.4}, the second
to last term is bounded, as we shall assume, if $N$ is large enough, while the last term is supported in the set
where $d_g(x,y)\in [\delta/2,\delta]$.  In practice we shall need to take $\delta>0$ to be small depending on
the potential $V$.

Using \eqref{h.2'} and taking adjoints we find that Lemma~\ref{lemmah1} yields the following:

\begin{prop}\label{proph1}
If $\delta>0$ is small we can write for $\la\ge1$
\begin{equation}\label{mainformula}I=T_\la \circ \bigl(-\Delta_g-(\la+i)^2\bigr) +R_\la,
\end{equation}
where $T_\la$ and $R_\la$ are integral operators with kernels $T_\la(x,y)$ and $R_\la(x,y)$, respectively, satisfying
\begin{equation}\label{h.9}
T_\la(x,y)=R_\la(x,y)=0 \quad \text{if } \, \, \, d_g(x,y)>\delta.
\end{equation}
Furthermore,
\begin{multline}\label{h.10}
T_\la(x,y)=\la^{\frac{n-3}2} e^{-i\la d_g(x,y)}
\, \bigl(d_g(x,y)\bigr)^{-\frac{n-1}2} \, a_\la(x,y),
\\ \text{if } \, \, \, d_g(x,y)\ge \la^{-1}, \quad \text{where } \, \, \, \nabla_{x,y}^\alpha a_\la(x,y)=O_\alpha\bigl((d_g(x,y))^{-|\alpha|})\bigr),
\end{multline}
and
\begin{multline}\label{h.11}
|T_\la(x,y)|\le C(d_g(x,y))^{2-n} \, \, \, \text{for } \, n\ge 3
\\ \text{and } \, \, |T_\la(x,y)|\le C|\log(\la d_g(x,y)/2)| \, \, \, \text{for } \, n=2, \quad
\text{if } \, \, d_g(x,y)\le \la^{-1}.
\end{multline}
Also, $R_\la(x,y)=r_\la(x,y)+b_\la(x,y)$ where $b_\la(x,y)$ is bounded independent of $\la\ge 1$ and
\begin{equation}\label{h.12}
r_\la(x,y)=\la^{\frac{n-1}2}e^{-i\la d_g(x,y)} c_\la(x,y),
\end{equation}
with
\begin{equation}\label{h.13}
|\nabla^\alpha_{x,y}c_\la|\le C_{\delta,\alpha}
\quad \text{and } \, \, c_\alpha(x,y)=0 \quad
\text{when } \, \, \, d_g(x,y)\notin [\delta/2,\delta].
\end{equation}
\end{prop}

The oscillatory integral operators with kernels $T_\la(x,y)$ ($d_g(x,y)\ge \la^{-1})$ and $r_\la(x,y)$ satisfy the
Carleson-Sj\"olin condition (see \cite{CarlesonSjolin} and \cite{SFIO2}).  Consequently, just as was done in
\cite{sogge88}, one can estimate the operators $T_\la$ and $R_\la$ using the oscillatory integral theorems
of H\"ormander~\cite{HormanderFLP} when $n=2$ and Stein~\cite{steinbeijing} when $n\ge 3$.

The bounds that we shall require are the following:

\begin{prop}\label{proph2}  Let $\sigma(p)$ be as in \eqref{a.3} and $p_c$ as in \eqref{a.6}
Then for $\la\ge 1$
\begin{equation}\label{h.14}
\|T_\la f\|_{L^p(M)}\le C_0\la^{\sigma(p)-1}\|f\|_{L^2(M)}
\end{equation}
and
\begin{equation}\label{h.15}
\|R_\la f\|_{L^p(M)}\le C_\delta \la^{\sigma(p)} \|f\|_{L^2(M)},
\end{equation}
provided that $p\in [p_c,\infty]$ if $n=2$ or $3$, $p\in [p_c,\infty)$ if $n=4$ and
$p\in [p_c,\tfrac{2n}{n-4}]$ for $n\ge 5$.
Here, $C_0=C_0(M,g)$ is an absolute constant, while $C_\delta=C(\delta,M,g)$ depends on
$\delta$.  Additionally, if $n\ge 3$
and $B_r(x_0)$ denotes the geodesic ball of small radius $r>0$ about $x_0\in M$, we have
for $p\in [p_c,\tfrac{2n}{n-3})$
\begin{equation}\label{h.16}
\|T_\la f\|_{L^p(B_\delta(x_0))}\le C_{r,p} \|f\|_{L^r(B_{2\delta}(x_0))}
\quad \text{where } \, \, \tfrac1r-\tfrac1p=\tfrac2n,
\end{equation}
with $C_{r,p}=C(M,g,r,p)$ independent of small $\delta>0$.
\end{prop}

The bounds \eqref{h.14} and \eqref{h.15} for the ``critical'' case where $p=p_c$ are in
\cite[Lemma 4.2]{sogge88}.  They are a consequence of the aforementioned oscillatory
integral bounds of H\"ormander~\cite{HormanderFLP} and Stein~\cite{steinbeijing}.    The proof for the
special case where $p=p_c$ is easily seen to handle the other exponents arising in
\eqref{h.14} and \eqref{h.15}.  The limitations on the exponents in higher dimensions is due
to the fact that convolution with $|x|^{2-n} {\bf 1}_{|x|\le 1} $ in dimensions $n\ge 4$ only maps
$L^2(\Rn)\to L^p(\Rn)$ for the exponents in the first part of the proposition if ${\bf 1}_{|x|\le1}$
denotes the indicator function of the unit ball in $\Rn$.  However, this is only used in the bounds \eqref{h.14} but not the bounds \eqref{h.15} on $R_\lambda$, which hold for $p \in [p_c, \infty]$ in any dimension $n \geq 2$.

Since $T_\la$ satisfies \eqref{h.10} and \eqref{h.11} with constants independent of $\delta$ we see that if
$p$ and $r$ are as in \eqref{h.16} then we see that
\begin{equation}\label{h.17}
\|T_\la\|_{L^r\to L^p} \le C_{r,p}
\end{equation}
by appealing to \cite[Proposition 2.2]{HuangSogge} or the earlier ``local'' bounds in Theorem 4.1 of \cite{DKS}.
Both are variable coefficient versions of the ``uniform Sobolev estimates'' \eqref{a.7} of Kenig, Ruiz and the third author \cite{KRS}.  As in
these estimates, the exponent $p$ must belong to $(\tfrac{2n}{n-1},\tfrac{2n}{n-3})$ and we note that $p_c$ lies in
this interval, which ensures that \eqref{h.17} is valid for $p\in [p_c,\tfrac{2n}{n-3})$ as in
\eqref{h.17}.  By \eqref{h.9} we immediately see that \eqref{h.17} implies
the localized variant \eqref{h.16}.

\section{Proof of Theorem~\ref{theorem1}}

To prove Theorem~\ref{theorem1}, we first notice that \eqref{a.4} trivially holds when $p=2$.  Based on this and a
simple interpolation using H\"older's inequality one finds that the special case of \eqref{a.4} where $p=p_c$ as
in \eqref{a.6} implies the bounds for $2<p\le p_c$.

As a result, we just need to prove \eqref{a.4} when $p\in [p_c, \tfrac{2n}{n-3})$.  This will allow us
to use \eqref{h.16}  in order to exploit our assumption that $V\in L^{\frac{n}2}(M)$.

To use this, for each small $\delta>0$ choose a maximal $\delta$-separated collection of points
$x_j\in M$, $j=1,\dots,N_\delta$, $N_\delta\approx \delta^{-n}$.  Thus, $M=\cup B_j$ if $B_j$ is the $\delta$-ball about
$x_j$, and if $B_j^*$ is the $2\delta$-ball with the same center,
\begin{equation}\label{3}
\sum_{j=1}^{N_\delta} {\bf 1}_{B^*_j}(x)\le C_M,
\end{equation}
where $C_M$ is independent of $\delta \ll 1$ if ${\bf 1}_{B_j^*}$ denotes the indicator function of $B^*_j$.  Since
$V\in L^{n/2}(M)$, we can fix $\delta>0$ small enough so that
\begin{equation}\label{4}
C_M\Bigl(\, C_0\sup_{x\in M}\|V\|_{L^{n/2}(B(x,2\delta))}\, \Bigr)^p <1/2,
\end{equation}
where $C_0$ is the constant in \eqref{h.14} above and \eqref{10} below.

Next, by \eqref{mainformula},
$$u(x)=T_\la\bigl((-\Delta_g-(\la+i)^2)u\bigr)(x)+R_\la u(x),$$
which we can rewrite as
$$u=T_\la\bigl((-\Delta_g+V-(\la+i)^2)u\bigr)+R_\la u-T_\la(Vu).$$
By \eqref{h.14}, \eqref{h.15} and \eqref{h.16} we can estimate the $L^p$ norms of each of the terms
over one of our $\delta$ balls as follows:
\begin{multline}\label{10}
\|u\|^p_{L^p(B_j)}
\le \bigl( C\la^{\sigma(p)-1}\|(-\Delta_g+V-(\la+i)^2)u\|_{L^2(M)}+ C_\delta \la^{\sigma(p)}\|u\|_{L^2(M)}\bigr)^p
\\
+\bigl(C_0\|Vu\|_{L^r(B_j^*)}\bigr)^p,
\end{multline}
Here, as in \eqref{h.16} the constant $C_0$ occurring in the last term, though depending on
$p$, is independent of $\delta$, and, moreover,
$$\frac1r=\frac{2}n+\frac1p.$$
Consequently, by H\"older's inequality,
\begin{equation*}
\|Vu\|_{L^r(B_j^*)}\le \|V\|_{L^{n/2}(B_j^*)} \, \|u\|_{L^p(B^*_j)}.
\end{equation*}
Combining this with \eqref{10} yields
\begin{multline}\label{ball}
\|u\|^p_{L^p(B_j)}
\le \bigl( C\la^{\sigma(p)-1}\|(-\Delta_g+V-(\la+i)^2)u\|_{L^2(M)} + C_\delta \la^{\sigma(p)}\|u\|_{L^2(M)} \bigr)^p
\\
+\bigl(C_0\|V\|_{L^{n/2}(B_j^*)} \|u\|_{L^p(B_j^*)}\bigr)^p,
\end{multline}

Since $M$ is the union of the $B_j$, and the number of these balls is $\approx \delta^{-n}$, if we add up the bounds in
\eqref{ball} and use \eqref{3} and \eqref{4} we get
\begin{align*}
\|&u\|^p_{L^p(M)}\le \sum_j \|u\|^p_{L^p(B_j)}
\\
&\le C'_\delta\bigl(  \la^{\sigma(p)-1}\|(-\Delta_g+V-(\la+i)^2)u\|_{L^2(M)}+\la^{\sigma(p)}\|u\|_{L^2(M)} \bigr)^p
\\
&\qquad\qquad\qquad\qquad\qquad\qquad\qquad\qquad
+\bigl(\sup_j C_0\|V\|_{L^{n/2}(B_j^*)}\bigr)^p\sum_j \|u\|^p_{L^p(B_j^*)}
\\
&\le C'_\delta\bigl(  \la^{\sigma(p)-1}\|(\Delta_g+V-(\la+i)^2)u\|_{L^2(M)} + \la^{\sigma(p)}\|u\|_{L^2(M)}\bigr)^p
\\
&\qquad\qquad\qquad\qquad\qquad\qquad\qquad\qquad
+C_M\bigl(\sup_j C_0\|V\|_{L^{n/2}(B_j^*)}\bigr)^p\|u\|^p_{L^p(M)}
\\
&\le C'_\delta\bigl( \la^{\sigma(p)-1}\|(-\Delta_g+V-(\la+i)^2)u\|_{L^2(M)} + \la^{\sigma(p)}\|u\|_{L^2(M)}\bigr)^p
+\tfrac12\|u\|^p_{L^p(M)},
\end{align*}
which of course implies \eqref{a.4}.

\section{Quasimode  estimates in three dimensions}\label{3dsection}

In this section we shall prove \eqref{a.11} when $n=3$.  As before, the special case where $p=p_c=4$  implies
the bounds for $2<p\le p_c$ and so we shall assume that $4\le p\le \infty$.

By taking adjoints in \eqref{mainformula}  (see also \eqref{h.2'}), we have
$$I=\bigl(-\Delta_g-(\la+i)^2\bigr) \circ T^*_\la +R^*_\la
\quad \text{on } \, \, C^\infty(M),$$
where the kernels of $T^*_\la$ and $R^*_\la$ are $T_\la(y,x)$ and $R_\la(y,x)$, respectively, with the
latter as in Proposition~\ref{proph1}.

To prove \eqref{a.11} it suffices to show that
\begin{multline}\label{3.3}
\bigl| \int u \, \psi \, dx \bigr|\le C_V \la^{\sigma(p)}\bigl(\la^{-1}\|(-\Delta_g+V-(\la+i)^2)u\|_2+\|u\|_2\bigr) \, +\, \tfrac12\|u\|_p,
\\ \text{for } \, \, \, \psi\in C^\infty(M) \, \, \, \text{with } \, \, \|\psi\|_{p'}=1,
\end{multline}
assuming that $u\in \Dom(H_V)$.  If we abbreviate the left side as $|(u,\psi)|$ then by the above
\begin{align}\label{3.4}
|(u,\psi)|&\le |(u,(-\Delta_g-(\la+i)^2)\circ T^*_\la \psi)|+|(u,R^*_\la \psi)|
\\
&\le \bigl| \bigr((-\Delta_g+V-(\la+i)^2)u, T^*_\la \psi\bigr)\bigr| +|(u,R^*_\la \psi)|+|(Vu,T^*_\la \psi)|. \notag
\end{align}

By duality, \eqref{h.14} yields $\|T^*_\la\|_{L^{p'}\to L^2}=O(\la^{\sigma(p)-1})$ and so
\begin{multline}\label{4.3} \bigl| \bigr((-\Delta_g+V-(\la+i)^2)u, T^*_\la \psi\bigr)\bigr|
\le \|(-\Delta_g+V-(\la+i)^2)u\|_2 \|T^*_\la \psi\|_2
\\
\le C\la^{\sigma(p)-1} \|(-\Delta_g+V-(\la+i)^2)u\|_2.
\end{multline}
Similarly, by \eqref{h.15}
\begin{equation}\label{4.4}
|(u,R^*_\la \psi)|\le \|u\|_2\|R^*_\la\psi\|_2\le C\la^{\sigma(p)}\|u\|_2.
\end{equation}
Thus, the left side of \eqref{3.3} is bounded by the first two terms in the right side plus
$$|(Vu,T^*_\la\psi)|.$$

To handle this, we shall use the fact that Sobolev embeddings give that $L^\infty(M)\subset \Dom(-\Delta_g+V)$
if $n=2,3$ (see \S 6). Thus, by \eqref{a.10}, \eqref{h.10} and \eqref{h.11}
\begin{equation}\label{4w}
T_\la(Vu)(x)=\int_M T_\la(x,y)\, V(y)\, u(y)\, dy, \quad
u\in \Dom(-\Delta_g+V),\end{equation}
is given by an absolutely convergent integral, as is $|(Vu,T^*_\la\psi)|.$  Hence, by Fubini's theorem
$$|(Vu,T^*_\la\psi)|=|(T_\la(Vu),\psi)|\le \|T_\la(Vu)\|_p\|\psi\|_{p'}=\|T_\la(Vu)\|_p.$$

For each fixed finite $p$, i.e., $p\in [4,\infty)$ we can repeat the arguments from the previous section to see
that if $1/r=1/p+2/3$, then by \eqref{h.16}, if $\delta>0$ is small enough, we can find a collection
of $\delta$-balls $B_j$ so that if $B^*_j$ is the double then
\begin{equation*}
\|T_\la(Vu)\|^p_{L^p(M)}\le \sum_j\|T_\la(Vu)\|_{L^p(B_j)}^p\le C_0^p\sum_j \|V\|^p_{L^{3/2}(B_j^*)}\|u\|^p_{L^p(B_j^*)}
\le 2^{-p}\|u\|_p^p.
\end{equation*}
This along with the earlier bounds for the first two terms in the right side of \eqref{3.4} yields \eqref{3.3}
for $p\in [4,\infty)$.

We cannot use this argument to handle the case where $p=\infty$ as \eqref{h.16} breaks down in this case.
On the other hand,
by Proposition~\ref{proph1}, $|T_\la(x,y)|\le C(d_g(x,y))^{-1}{\bf 1}_{d_g(x,y)<\delta}(x,y)$ and so the
Kato condition \eqref{a.10} ensures that
$$\|T_\la(Vu)\|_\infty\le \tfrac12\|u\|_\infty$$
if $\delta>0$ is small enough.  By the above this implies that \eqref{3.3} is also valid when $p=\infty$ and
$n=3$, which completes the proof of the three-dimensional results in Theorem~\ref{theorem2}.

%
%

\section{Quasimode estimates in two dimensions}\label{2dsection}

Let us now prove the estimates in Theorem~\ref{theorem2} when $n=2$.  This is a unique case
since the off-diagonal uniform Sobolev estimates of Kenig, Ruiz and the third author \cite{KRS}
do not hold in two dimensions.  Consequently, we cannot use an inequality like
\eqref{h.16} when $n=2$.

Fortunately, we can prove \eqref{a.11} when $p=\infty$ and $n=2$ exactly as before since
we are assuming that $V\in {\mathcal K}(M)$.

To see this, we argue as in the preceding section to see that it is enough to prove \eqref{3.3} for
$p=\infty$ in order to obtain \eqref{a.11} for this exponent.

As before, \eqref{h.14} and \eqref{h.15} yield \eqref{4.3} and \eqref{4.4}, respectively for all
$p\in [p_c,\infty]=[6,\infty]$.  This means that for all such exponents the first two terms in the right
side of \eqref{3.4} are dominated by the first two terms in the right side of \eqref{3.3}.
Also, as before, $\Dom(-\Delta_g+V)\subset L^\infty(M)$ and so the Kato condition ensures that
$T_\la(Vu)$ is given by an absolutely convergent integral.  Thus, we would have \eqref{3.3} for
$p=\infty$ if we could choose $\delta>0$ so that
$$\|T_\la(Vu)\|_\infty \le \tfrac12\|u\|_\infty.$$
This follows exactly as before due to the fact that by Proposition~\ref{proph1}
$$|T_\la(x,y)|\le C\, h_2(d_g(x,y)) \, {\bf 1}_{d_g(x,y)<\delta}(x,y).$$

To finish the proof of the two-dimensional results in Theorem~\ref{theorem2}, it is now enough
to prove \eqref{3.3} when $p=p_c=6$ since this yields \eqref{a.11} for this case, and, by H\"older's
inequality the remaining cases follow from this, the trivial case where $p=2$ and the case where
$p=\infty$ that we just proved.

By the fact that the first two terms in the right side of \eqref{3.4} are under control for this exponent and the
above arguments it is enough to bound
$$\|T_\la(Vu)\|_6.$$
Unlike all the earlier arguments we cannot bound this by $\tfrac12\|u\|_6$ due to the aforementioned
fact that we cannot appeal to \eqref{h.16}.

To get around this, we shall use the fact that Proposition~\ref{proph1} yields
\begin{equation*}
|T_\la(x,y)|\le
\begin{cases} C_0\la^{-1/2}\bigl(d_g(x,y)\bigr)^{-1/2}, \quad \text{if } \, \, d_g(x,y)\ge \la^{-1}
\\
C_0|\log(\la d_g(x,y)/2)|, \quad \text{if } \, \, d_g(x,y)\le \la^{-1}.
\end{cases}
\end{equation*}
As a result,
$$\sup_y \bigl(\int_M |T_\la(x,y)|^6\, dx\bigr)^{1/6} \le C\la^{-1/3}.$$
Whence, by Minkowski's integral inequality,
$$\|T_\la(Vu)\|_6\le C\la^{-1/3}\|Vu\|_1\le C\la^{-1/3}\|V\|_1 \|u\|_\infty.$$
Since we are assuming that $V\in L^1(M)$ and we just proved that
$$\|u\|_\infty \lesssim \la^{1/2}\bigl(\la^{-1}\|(-\Delta_g+V-(\la+i)^2)u\|_2 +\|u\|_2\bigr),$$
we conclude that $\|T_\la(Vu)\|_6$ is also dominated by the first two terms in the right side of \eqref{3.3}, which finishes
the proof.

\section{Remaining bounds for higher dimensions}\label{heatsection}

In this section we shall prove the bounds in Theorem~\ref{theorem2} for $n\ge 4$.  Since $L^{n/2}\subset L^2$ for
$n\ge4$, it follows that $C^\infty$ is an operator core for $H_V$ (see \cite{SimonSurvey}) and so
Theorem~\ref{theorem1} and the spectral theorem implies that \eqref{a.11} is valid when $2<p<\tfrac{2n}{n-3}$.
So it remains to prove this inequality in higher dimensions when $\tfrac{2n}{n-3}\le p<\tfrac{2n}{n-4}$, as well
as \eqref{a.12} for the remaining cases where $p\in [\tfrac{2n}{n-4},\infty]$.  We shall conclude the section by showing
that when $V\equiv 0$ \eqref{a.11} breaks down on any manifold if $p=\infty$ and $n=4$ or $p>\tfrac{2n}{n-4}$
and $n\ge 5$.

To prove the positive results, now let $R_\la: \, L^2\to L^2$ denote the spectral projection operator
corresponding to the interval $(2\la,\infty)$, i.e., $R_\la ={\bf 1}_{P_V>2\la}$,
so that
$$R_\la f=\sum_{\la_j >2\la} \langle f,e_j\rangle e_j,$$
where $\{e_j\}$ is an orthonormal basis of eigenfunctions of $P_V$ with eigenvalues $0<\la_1\le \la_2\dots$.
Recall that in \S~\ref{katosection} we argued the the spectrum of $P_V$ is discrete.

Using probabilistic methods, specifically the Feynman-Kac formula, this yields the same sort of bounds for $e^{-tH_V}$
since $V\in {\mathcal K}(M)$, i.e.
\begin{equation}\label{5.3'}
\|e^{-tH_V}\|_{L^p(M)\to L^q(M)}\lesssim t^{-\frac{n}2(\frac1p-\frac1q)}, \quad \text{if } \, \,
0<t\le 1, \, \, \, \text{and } \, 1\le p\le q\le \infty.
\end{equation}
To prove these, we shall use the fact that, if $V$ is in the Kato class, Sturm~\cite[Theorem 4.12]{Sturm} proved that the kernel
of $e^{-tH_V}$ satisfies the pointwise
bounds in \eqref{s.2}.  This implies that the bounds for $p=1$ and $q=\infty$ are valid, as well as the case
where $p=q=\infty$.  Since, by the spectral theorem, the heat operator is also uniformly bounded on $L^2$
when $0<t\le 1$, one gets the remaining cases in \eqref{5.3'} by interpolation.
See also the later work of Stollmann and Voigt~\cite[Theorem 5.1]{StollmannVoigt} and G\"uneysu \cite{Guneysu} for such results in a more general setting.
In Aizenman and Simon~\cite{AZ} it was shown that
one needs the assumption that $V\in {\mathcal K}(\Rn)$ to get reasonable heat operator bounds in the Euclidean setting and their arguments
extend to our setting.  Before the aforementioned results, Aizenman and Simon~\cite{AZ}  also showed that the
bounds in \eqref{5.3'} are valid for $\Delta+V$
in $\Rn$ if $V\in {\mathcal K}$, which we shall use in the final section.

Next, let $L_\la={\bf 1}_{P_V\le 2\la}$ denote the projection onto frequencies $\le 2\la$ so that
$I=L_\la +R_\la$ if $R_\la$ is as above.
We then claim that we can use the special case of \eqref{a.11} corresponding to $p=p_c$
along with
\eqref{5.3'} to prove
\begin{equation}\label{5.1'}
\|L_\la u\|_{L^p(M)}
\lesssim_{V,M} \la^{\sigma(p)-1}\|(-\Delta+V-(\la+i)^2)u\|_{L^2(M)}, \quad \text{if } \, \, p>p_c.
\end{equation}

To prove \eqref{5.1'} let us fix a nonnegative function $\beta\in C^\infty_0((1/2,1))$
satisfying $\int_{-\infty}^\infty \beta(t)\, dt=1$
and consider the Laplace transform
of the following $L^1$-normalized dilates of $\beta$
$$\tilde \beta_\la(\tau)=\int_0^\infty e^{-t\tau} \la^2 \beta(\la^2 t) \, dt, \quad \tau\ge 0.$$
Clearly, we have
$$C_0^{-1}\le \tilde \beta_\la(\tau)\le C_0, \quad \text{if } \, \, 0\le \tau \le 4\la^2,$$
for some uniform constant $C_0<\infty$.
As a result, by the spectral theorem, the operator
$\tilde L_\la f=\sum_{\la_j\le 2\la} \bigl(\tilde \beta_\la(\la_j^2)\bigr)^{-1}\langle f,e_j\rangle e_j$ satisfies
\begin{equation}\label{5.4}
\|\tilde L_\la\|_{L^2\to L^2}\le C_0, \quad
\text{and } \, \, \tilde \beta_\la(H_V)\circ \tilde L_\la =L_\la.
\end{equation}
Since
$$ \tilde \beta_\la(H_V)=\int_0^\infty e^{-tH_V} \, \la^2 \, \beta(\la^2 t)\, dt,$$
by \eqref{5.3'} we have the following bounds for these ``Bernstein-type'' operators
\begin{equation}\label{5.5}
\bigl\|  \tilde \beta_\la(H_V)\|_{L^p\to L^q}\lesssim \la^{n(\frac1p-\frac1q)}, \quad
\text{if } \, \, 2\le p\le q\le \infty.
\end{equation}

If we use the second part of \eqref{5.4}, \eqref{5.5} and the special case of \eqref{a.11} corresponding to $p=p_c$, we conclude that for $p>p_c$ we have
$$\|L_\la u\|_p\lesssim \la^{n(\frac1{p_c}-\frac1p)} \|\tilde L_\la u\|_{p_c}
\lesssim  \la^{\frac1{p_c}-1} \la^{n(\frac1{p_c}-\frac1p)}
\|(-\Delta_g+V-(\la+i)^2)\tilde L_\la u\|_{L^2(M)}.$$
Since $\sigma(p)=\tfrac1{p_c}+n(\tfrac1{p_c}-\tfrac1p)$ and
$$\|(H_V-(\la+i)^2)\tilde L_\la u\|_{L^2(M)}=\|\tilde L_\la(H_V-(\la+i)^2) u\|_{L^2(M)}
\le C_0 \|(H_V-(\la+i)^2) u\|_{L^2(M)},$$
by the first part of \eqref{5.4}, we obtain \eqref{5.1'}.

Using \eqref{5.5} and the spectral theorem also gives
\begin{equation}\label{6.5}
\|R_\la u\|_{L^p(M)}\lesssim \|(-\Delta_g+V)R_\la u\|_{L^2(M)}\lesssim
\|(-\Delta_g+V-\la+i)^2)u\|_{L^2(M)} \quad \text{if } \, p<\tfrac{2n}{n-4}.
\end{equation}
Since $\sigma(p)-1\ge 0$ when $p\ge \tfrac{2n}{n-3}$, this along with \eqref{5.1'} yields \eqref{a.11} when
$p\in [\tfrac{2n}{n-3},\tfrac{2n}{n-4})$.

For the remaining case, we note that \eqref{5.5} implies that $\|R_\la u\|_p$ is dominated by the last
term in the right side of \eqref{a.12}.  This along with \eqref{5.1'} gives us \eqref{a.12}, which finishes
the proof of the estimates in higher dimensions.

\bigskip

Let us conclude this section by showing that \eqref{a.11} need not hold if $p=\infty$ and $n=4$ or
$p>\tfrac{2n}{n-4}$ and $n\ge 5$.  We shall adapt the arugment in \cite[pp. 164--165]{SoggeTothZelditch}.

To prove these negative results we recall the local Weyl formula which says that for large $\mu$ we have
\begin{equation}\label{6.6}
\sum_{\la_j\le \mu}|e_j(x_0)|^2 \approx \mu^n, \quad \forall \, x_0\in M.
\end{equation}

To use this, fix a nonnegative Littlewood-Paley bump function $\beta\in C^\infty_0((1/2,2))$ satisfying
$1=\sum_{k=-\infty}^\infty \beta(r/2^k)$, $r>0$.  Then if we assume that $\la^2$ is an eigenvalue of
$-\Delta_g$, choose an eigenfunction $e_\la$ satisfying $\|e_\la\|_2=1$, fix $x_0\in M$ and set for
$0<\e <1/2$
$$u_\la(x)=e_\la(x)+\sum_{2^k\ge \la}2^{-(\frac{n}2+2)k}k^{-\frac12-\e}\,
\bigl(\beta(P/2^k)\bigr)(x,x_0).$$
Here $P=\sqrt{-\Delta_g}$, and
$$\bigl(\beta(P/2^k)\bigr)(x,y)=\sum \beta(\la_j/2^k)e_j(x) e_j(y)$$
is the kernel of the operator $\beta(P/2^k)$.  By \eqref{6.6}
\begin{equation}\label{4.1}\bigl\| \bigl(\beta(P/2^k)\bigr)(\, \cdot \, ,x_0)\bigr\|_2^2=\sum_{j=0}^\infty \beta^2(\la_j/2^k)|e_j(x_0)|^2
\approx 2^{nk},
\end{equation}
and, similarly, since $\beta\ge 0$,
\begin{equation}\label{4.2}\bigl(\beta(P/2^k)\bigr)(x_0,x_0)\approx 2^{nk}.\end{equation}

Since $\beta(r/2^k)\beta(r/2^\ell)=0$ when $|k-\ell|\ge 10$, we conclude from \eqref{4.1} that
$$
\|u_\la-e_\la\|_2^2\lesssim\sum_{2^k\ge \la}2^{-(n+4)k}k^{-1-\e} \, 2^{nk} =o(1),$$
and so
$$\|u_\la\|_2=1+o(1).$$
Similarly, since $(\Delta_g+\la^2)e_\la=0$ and
$$(\Delta_g+(\la+i)^2)\bigl(\beta(P/2^k)\bigr)(x,x_0)=\sum_{j=0}^\infty (-\la_j^2+(\la+i)^2) \,\beta(\la_j/2^k)e_j(x)e_j(x_0),$$
we see from \eqref{4} that
\begin{align*}
\|(\Delta_g+(\la+i)^2)u_\la\|^2_2&\lesssim\sum_{2^k\ge \la} 2^{-(n+4)k}2^{4k}k^{-1-2\e} \, \sum_{\la_j\approx 2^k}
|e_j(x_0)|^2
\\
&\lesssim \sum_{2^k\ge \la}k^{-1-2\e}=o(1).
\end{align*}
Thus
\begin{equation}\label{6.9}
\la^{-1}\|(-\Delta_g-(\la+i)^2)u_\la\|_2 +\|u_\la\|_2 \approx 1.
\end{equation}

On the other hand, by \eqref{4.2} and the fact that $\beta\ge 0$, we obtain
\begin{align*}
u_\la(x_0)-e_\la(x_0)&=\sum_{\la\ge 2^k}
2^{-(\frac{n}2+2)k}k^{-\frac12-\e} \bigl(\beta(P/2^k)\bigr)(x_0,x_0)
\\
&\approx \sum_{2^k\ge \la} 2^{-(\frac{n}2+2)k}2^{nk}k^{-\frac12-\e}
\approx \sum_{2^k\ge \la}2^{(\frac{n}2-2)k}k^{-\frac12-\e}=\infty,
\end{align*}
if $n\ge 4$ since $0<\e<1/2$.  Since, by results in \cite{sogge88}, $\|e_\la\|_\infty =O(\la^{\frac{n-1}2})$, we conclude
from this that $u_\la\notin L^\infty$ and hence \eqref{a.11} need not hold for $p=\infty$ for such $n\ge4$.

It is straightforward to modify this argument to show that \eqref{a.11} need not hold as well when
$\tfrac{2n}{n-4}<p<\infty$ if $n\ge 5$.  For such $p$ and $n$  for small $\e>0$ let
$$u_\la(x)=e_\la(x)+\bigl( P^{-(\tfrac{n}2+2+\e)} \, \rho(P/\la)\bigr)(x,x_0)$$
where $\rho\in C^\infty$ vanishes near $0$ but equals one near infinity.

By arguing as before it is not difficult to check that as $\la \to \infty$
$$\|u\|_2\approx 1 \quad \text{and } \, \,
\|(\Delta_g+\la^2)u_\la\|_2 =o(1).$$
Furthermore, by arguing as in Chapter 4 of \cite{SFIO2} it is also straightforward to verify that,
if $\text{dist}(x,x_0)\lesssim \la^{-1}$, we have
$$|u_\la(x)-e_\la(x)|\approx  \, \bigl(\text{dist}(x,x_0)\bigr)^{-\frac{n}2+2+\e}.$$
Since the right side is not in $L^p$ of a ball of radius $\approx\la^{-1}$ about $x_0$ if $p>\tfrac{2n}{n-4}$ and
$\e>0$ is sufficiently small, we conclude that there are $u_\la$ satisfying \eqref{6.9}
and $u_\la\notin L^p(M)$ for such $p$ if $n\ge5$, which shows that \eqref{a.11} need not hold in this case, as claimed.

\section{Applications to spectral multipliers}\label{specmult}
Let $\chi_\lambda^V$ be the projection operator $\chi_\lambda^V f = \sum_{\lambda_j \in [\lambda, \lambda +1]}\langle f,e_{j}\rangle e_{j}$ as defined above.  In this section, we examine the consequences of Corollary 1.4 for some of the spectral multiplier theorems of significance in harmonic analysis, in particular, estimates for Bochner-Riesz means and the H\"ormander multiplier theorem. To this end it is helpful to observe the counterparts of our main theorem in dual spaces for $1 \leq p \leq \frac{2(n+1)}{n+3}$:
\begin{align}
  \|\chi_\lambda^V\|_{L^p(M) \to L^2(M)} &\lesssim \lambda^{\frac np - \frac{n+1}{2}}, \label{chilambdadual}\\
  \|\chi_\lambda^V\|_{L^p(M) \to L^{p'}(M)} &\lesssim \lambda^{n(\frac 1p-\frac{1}{p'}) - 1} . \label{chilambdaTT}
\end{align}
Much of the early motivation for developing these bounds when $V=0$ emerged from their applications to spectral multipliers.  In particular, in \cite{sogge87}, the third named author used the $L^p$ bounds in  \cite{sogge88} to give optimal bounds on Bochner-Riesz means in $L^p$ spaces when $\min(p,p') \in [1,\frac{2(n+1)}{n+3}]$.  The work \cite{sogge02} then expounded on this relation, clarifying the role of finite speed of propagation for the wave equation in such results, thus giving a means approaching cases where boundary conditions are nontrivial.  Moreover, in \cite{SeegerSoggePDOfunction}, Seeger and the third author used such bounds to extend the H\"ormander multiplier theorem \cite{HormanderMultiplier} to
functions of self-adjoint elliptic
 pseudodifferential operators on compact manifolds.

Recall for operators with nonnegative discrete spectrum, the Bochner-Riesz means $S_\lambda^\delta$ are defined by
\begin{equation*}
  S_\lambda^\delta f = \sum_{\lambda_j \leq \lambda} \left(1-\lambda_j^2/\lambda^2\right)^\delta \langle f,e_{j}\rangle e_{j}.
\end{equation*}
As before, throughout this section, without loss of generality, we shall assume that $H_V$ is positive.
A well-known necessary condition for $S_\lambda^\delta$ to be bounded on $L^p$ is that $\delta > \delta(p)$ where
  \begin{equation}\label{BRexponent}
    \delta(p) = \max(n|1/2-1/p|-1/2,0).
  \end{equation}
We now state the consequences of our main results for $L^p$ boundedness of $S_\lambda^\delta $ and the H\"ormander multiplier theorem.

\begin{theorem}\label{T:BR}
  Let $V \in L^{n/2}(M)\cap \K(M)$.  Suppose $\min(p,p') \in [1,\frac{2(n+1)}{n+3}]$ and that $\delta(p)$ is given by \eqref{BRexponent}.   Then for any $\delta>\delta(p)$, $S_\lambda^\delta$ is uniformly bounded on $L^p$.  That is, there exists a constant $C$ independent of $\lambda$ such that
  \begin{equation}\label{mainBRbound}
    \|S_\lambda^\delta \|_{L^p(M) \to L^p(M)} \leq C.
  \end{equation}
\end{theorem}

\begin{theorem}\label{T:HormanderMult}
  Let $V \in L^{n/2}(M)\cap \K(M)$.  For $1 < r < \infty$, set $r^* = \min(r,r')$.  Suppose $m \in L^\infty(\R)$ satisfies
  \begin{equation*}
   \sup_{\mu >0} \|\beta(\cdot)m(\mu \cdot)\|_{H^{s}(\R)} < \infty, \qquad \text{ where }s>\max\left(n\left(\frac 1{r^*} - \frac 12\right),\frac 12\right),
  \end{equation*}
  whenever $\beta \in C^\infty_0((1/2,2))$.
  Then $m(\sqrt{H_V})$ is bounded on $L^r(M)$.
\end{theorem}

Theorems \ref{T:BR} and \ref{T:HormanderMult} are a consequence of Corollary 1.4, \eqref{chilambdadual}, \eqref{chilambdaTT}, and finite speed of propagation for the corresponding wave equation, namely \eqref{finitespeedconclusion2} below.  Indeed, once the latter is observed, the aforementioned method in \cite{sogge02} proves Theorem \ref{T:BR} with no essential change in the proof.  Moreover, results of Chen, Ouhabaz, Sikora, and Yan \cite[Theorem C(ii)]{ChenOuhabazSikoraYanRestrictionSpectralMultipliers} give rather general sufficient conditions which ensure the Bochner-Riesz means associated to a nonnegative self-adjoint operator satisfy \eqref{mainBRbound}, and these conditions are satisfied here.  In particular, in Proposition I.14 of that work, the authors show that \eqref{chilambdaTT} is enough to imply that the crucial condition ``$\text{SC}_{p,2}^{2,1}$'' on p.225 is satisfied for $p \in [1,\frac{2(n+1)}{n+3}]$.  These same hypotheses (finite propagation speed and $\text{SC}_{p,2}^{2,1}$) are also enough to yield Theorem \ref{T:HormanderMult} given Theorem B there.

\begin{remark}
It is now known that slightly weaker versions of the H\"ormander multiplier theorem follow from heat kernel methods.  In particular, Alexopoulos \cite[Theorem 6.1]{AlexopoulosSpectralMultipliersMarkov} showed that whenever the heat kernel satisfies Gaussian upper bounds of the form \eqref{s.2}, then $m(\sqrt{H_V})$ is bounded on $L^r(M)$ for $1<r< \infty$ provided the stronger hypothesis
\begin{equation}\label{AlexopHyp}
\sup_{\mu >0} \|\beta(\cdot)m(\mu \cdot)\|_{C^{s}(\R)} < \infty, \qquad s >n/2,
\end{equation}
is satisfied. See also \cite{DuongOuhabazSikora} for results of this type. As noted in \S\ref{heatsection}, results of Sturm \cite[Theorem 4.12]{Sturm} give these Gaussian upper bounds.  Strictly speaking, the hypotheses of Alexopoulos require uniform upper bounds $p_t(x,y) \leq C$ when $t \geq 1$. However, this can be achieved by replacing $V$ by $V+N$ for $N$ large enough as in \S\ref{kato}, since by the spectral theorem, this has the effect of multiplying the heat kernel by $e^{-Nt}$.  The hypothesis \eqref{AlexopHyp} is satisfied by the multipliers which yield the usual bound for the Littlewood-Paley square function, see \cite[Theorem 5, Ch. IV]{SteinSingularIntegrals}.  Among other things, this bound can be used to see that \eqref{6.5} is satisfied at the endpoint $p=\frac{2n}{n-4}$
if $n\ge5$.
\end{remark}


\begin{corollary}
  Let $V \in L^{n/2}(M)\cap \K(M)$ and let $1 < r < \infty$.  Let $\{\beta_j\}_{j \geq 0}$ be a sequence of bump functions on $\R$ satisfying $\beta_0(\xi) + \sum_{j=1}^{\infty} \beta_j(\xi) \equiv 1$ where $\beta_j(\xi) = \beta_1(2^{1-j}\xi)$, with $\supp(\beta_1) \subset \{|\xi| \in (\frac 12,2)\}$ and $\supp(\beta_0)\subset\{|\xi| \in (-\infty,1)\}$.  Define the Littlewood-Paley square function for $f \in L^r(M)$ by
  \begin{equation*}
    Sf = \left(\sum_{j=0}^{\infty}\left|\beta_j(\sqrt{H_V}) f \right| \right)^{1/2}.
  \end{equation*}
  Then there exists uniform constants $c_r,C_r$ such that $c_r \leq \|Sf\|_{L^r(M)}/\|f\|_{L^r(M)} \leq C_r$ for $f \neq 0$ in $L^r(M)$.
\end{corollary}


Given the above, Theorems \ref{T:BR} and \ref{T:HormanderMult} are now a consequence of the following lemma.
\begin{lemma}[Finite propagation speed]\label{L:finitespeed}
  Suppose $H_V = -\Delta_g + V$, with $V \in L^{\frac n2}(M)\cap \K(M)$. Suppose $u,v \in L^2(M)$ satisfy
  $d_g(\supp(u),\supp(v)) = R$, then
\begin{equation}\label{finitespeedconclusion}
  \langle u, \cos(t \sqrt{H_V}) v \rangle =0, \qquad |t| \leq R.
\end{equation}
Consequently, if $\cos(t \sqrt{H_V})(x,y)$ denotes the integral kernel of $\cos(t \sqrt{H_V})$,
\begin{equation}\label{finitespeedconclusion2}
  \supp\left(\cos(t \sqrt{H_V})(\cdot,\cdot)\right) \subset \left\{(x,y) \in M \times M : d_g(x,y) \leq |t| \right\}.
\end{equation}
\end{lemma}
When $n \geq 5$, this is a consequence of results of Chernoff \cite[Proposition 4.3]{Chernoff}.  In particular, it is shown that if $V \in L^q(M)$ with $q=2$ when $n \leq 3$, $q>2$ when $n=4$, and $q = n/2$ when $n \geq 5$, then \eqref{finitespeedconclusion}, \eqref{finitespeedconclusion2} are satisfied.  Similarly, Remling\footnote{Strictly speaking, Remling's work considers constant coefficient Laplacians rather than the Laplace-Beltrami operator considered here, but the arguments extend to our setting by standard energy estimates for the wave equation.}  \cite[Lemma 2.2]{Remling} observed that this holds whenever $C^\infty(M)$ is an operator core for $H_V$.  The argument below instead uses form cores for the quadratic forms defined by $V$, which are equivalent to operator cores for $\sqrt{H_V}$ (see e.g. \cite[p.606]{SimonComprehensiveIV}).

Note that \eqref{finitespeedconclusion2} is a consequence of \eqref{finitespeedconclusion} by typical measure theoretic considerations: if
\begin{equation*}
  \supp\left(\cos(t \sqrt{H_V})(\cdot,\cdot)\right) \cap \{d_g(x,y) > |t| \}
\end{equation*}
has positive measure, then one can find $u,v \in L^2(M)$ for which \eqref{finitespeedconclusion} fails to hold.

\begin{proof}
We first observe that if $V \in L^\infty(M)$, then  \eqref{finitespeedconclusion} holds by the usual energy estimates and Gronwall's inequality (see e.g. \cite[Lemma 2.3]{Remling}), or by the previously cited works \cite{Chernoff}, \cite{Remling}.
For integers $M,N \geq 1$, define
\begin{equation*}
  V_N^M :=  \mathbf{1}_{\{ -N<V(x)<M\}}V, \qquad V_N := \mathbf{1}_{\{ -N<V(x)\}}V.
\end{equation*}
Let $Q$ denote the quadratic form associated to $H_V$ and define the approximating forms
\begin{equation}\label{quadformsdef}
  Q_N^M(w) := \int_M |\nabla_g w|^2 + V_N^M |w|^2\,dx, \quad Q_N(w) := \int_M |\nabla_g w|^2 + V_N |w|^2\,dx,
\end{equation}
which generate corresponding operators $H_{V_N^M},H_{V_N}$ respectively.  In all cases, we assume the quadratic forms assume the value $+\infty$ whenever $w$ is not in the domain of the form.

We appeal to the monotone convergence theorem for forms in \cite[Theorem 7.5.18]{SimonComprehensiveIV}, using parts (a) and (b) for increasing and decreasing sequences respectively.  Since $Q_N^M(w) \leq Q_N^{M+1}(w)$, for each $M$ part (a) of the theorem yields strong resolvent convergence:
\begin{equation*}
  \lim_{M\to\infty}\|(H_{V_N^M} \pm i)^{-1} w-(H_{V_N} \pm i)^{-1} w\| =0 \quad\text{ for all } w\in L^2(M).
\end{equation*}
By \cite[Theorem VIII.20b]{ReedSimon} or \cite[Theorem 7.2.10]{SimonComprehensiveIV}, strong resolvent convergence implies the strong convergence $f(H_{V_N^M}) \to f(H_{V_N})$ as $M\to\infty$ for any bounded continuous $f$ on $\R$.  In particular, $\cos(t \sqrt{H_{V_N^M}}) \to \cos(t \sqrt{H_{V_N}})$ strongly which implies that since \eqref{finitespeedconclusion} holds for each $H_{V_N^M}$, it persists in the limit and is satisfied by $H_{V_N}$.


We now conclude the proof by taking limits as $N \to \infty$.  Since $Q_{N+1}(w) \leq Q_N(w)$ for each $w$, part (b) of the monotone convergence theorem for forms gives the strong convergence
\begin{equation*}
(H_{V_N} \pm i)^{-1} \to (H_{V} \pm i)^{-1} , \quad \text{ and hence } \cos\left(t \sqrt{H_{V_N}}\right) \to \cos(t \sqrt{H_{V}}) .
\end{equation*}
As before \eqref{finitespeedconclusion} thus persists in the limit.
\end{proof}

\section{Strichartz estimates for the wave equation}\label{strichartzsection}

Let us now see how the spectral projection estimates \eqref{a.13} in Corollary~\ref{spec} can also be used to prove natural
Strichartz estimates for $H_V=-\Delta_g+V$.  As above, without loss of generality, we shall assume that $H_V\ge 0$.

\begin{theorem}\label{strthm}  Let $(M,g)$ be a compact manifold of dimension $n\ge2$ and assume that
$V\in L^{n/2}(M)\cap {\mathcal K}(M)$.  Let $u$ be the solution of
\begin{equation}\label{8.1}
\begin{cases}
\bigl(\partial_t^2-\Delta_g+V(x)\bigr)u=0
\\
u|_{t=0}=f_0, \quad \partial_tu|_{t=0}=f_1.
\end{cases}
\end{equation}
Then
\begin{equation}\label{8.2}
\|u\|_{L^{\frac{2(n+1)}{n-1}}([0,1]\times M)}\le C_V\bigl(\|(I+P_V)^{1/2}f_0\|_{L^2(M)}
+\|(I+P_V)^{-1/2}f_1\|_{L^2(M)}\bigr),
\end{equation}
with $P_V$ denoting $\sqrt{H_V}$.
\end{theorem}

\begin{remark}
This exactly corresponds to the original $L^{\frac{2(n+1)}{n-1}}(\R\times \Rn)$ estimate of
Strichartz~\cite{Strichartz77} for the wave equation.  Indeed, if $M=\Rn$ and $H_V$ is the standard
Laplacian, the analog of \eqref{8.2} along with a scaling argument yields
\begin{equation}\label{8.3}
\|u\|_{L^{\frac{2(n+1)}{n-1}}(\R\times \Rn)}\lesssim \|f_0\|_{\dot H^{1/2}(\Rn)}+\|f_1\|_{\dot H^{-1/2}(\Rn)}.
\end{equation}
The variant of \eqref{8.2} with $V\equiv0$ can be proved using parametrices as was done by
Kapitanski~\cite{Kapitanski89} and Mockenhaupt, Seeger and the third author \cite{MSS} and this special case
of \eqref{8.2} is seen to yield the classical Strichartz estimate \eqref{8.3}.  Of course the existence of
eigenfunctions imply that, unlike \eqref{8.3}, on $(M,g)$ one cannot have the analog of \eqref{8.2}
where the norm in the left is taken over $\R\times M$.\end{remark}

\begin{proof}[Proof of Theorem~\ref{strthm}]  In \cite[Theorem 2.1]{blpcritical}, the authors show that the bounds \eqref{8.2} follows from Corollary \ref{spec}, and their proof works equally well in our circumstances.  
See also \cite{Nicola}.
Nonetheless, we include a proof for the sake of completeness which will serve as a model for certain global Strichartz
estimates that we shall obtain in $\Rn$ in \S\ref{Rn}.

If, as above, $p_c=\tfrac{2(n+1)}{n-1}$,
then to prove \eqref{8.2} it suffices to show
that
\begin{equation}\label{8.4}
\bigl\|e^{itP_V}f\bigr\|_{L^{p_c}([0,1]\times M)} \lesssim \|(I+P_V)^{1/2}f\|_{L^2(M)}.
\end{equation}
To prove this, it suffices to prove that whenever we fix $\rho\in {\mathcal S}(\R)$ satisfying
$\text{supp }\Hat \rho\subset (-1/2,1/2)$ we have
\begin{equation}\label{8.4'}\tag{8.4$'$}
\bigl\|\rho(t)e^{itP_V}f\|_{L^{p_c}(\R\times M)} \lesssim \|(I+P_V)^{1/2}f\|_{L^2(M)}.
\end{equation}

To prove this, we shall change notation a bit and let
$$\chi^V_kf=\sum_{\la_j\in (k,k+1]}E_jf, \quad E_jf=\langle f, e_{\la_j}\rangle e_{\la_j},$$
so that
$f=\sum_{k=0}^\infty \chi_k^Vf$.  Then, since $\sigma(p_c)=1/p_c$, \eqref{a.12} yields
\begin{equation}\label{8.5}
\|\chi_k^Vf\|_{L^{p_c}(M)}\lesssim (1+k)^{1/p_c}\|f\|_{L^2(M)}, \quad k=0,1,2,\dots .
\end{equation}

To use this, we first note that by Sobolev estimates
$$\bigl\|\rho(t)e^{itP_V}f\|_{L^{p_c}(\R\times M)} \lesssim
\bigl\|\, |D_t|^{1/2-1/p_c}\bigl(\rho(t)e^{itP_V}f\bigr)\bigr\|_{L^{p_c}_xL^2_t(\R\times M)}.$$
If we let
$$F(t,x)= |D_t|^{1/2-1/p_c}\bigl(\rho(t)e^{itP_V}f(x)\bigr)$$
denote the function inside the mixed-norm in the right, then
$$F(t,x)=\sum_{k=0}^\infty F_k(t,x),$$
where
$$F_k(t,x)= |D_t|^{1/2-1/p_c}\bigl(\rho(t)e^{itP_V}\chi_k^Vf(x)\bigr).$$
Consequently, its $t$-Fourier transform is
\begin{equation}\label{8.6}\Hat F_k(\tau,x)=|\tau|^{1/2-1/p_c}\sum_{\la_j\in [k,k+1)} \Hat \rho(\tau-\la_j)E_jf(x).\end{equation}
Since we are assuming $\text{supp }\Hat \rho \subset (-1/2,1/2)$, we conclude that
$$\int_{-\infty}^\infty F_k(t,x)\, \overline{F_\ell(t,x)}\, dt =(2\pi)^{-1}\int_{-\infty}^\infty
\Hat F_k(\tau,x)\, \overline{\Hat F_\ell(\tau,x)} \, d\tau =0
\quad \text{when } \, |k-\ell|>10.$$
As a result
\begin{multline*}
\bigl(\, \int_{-\infty}^\infty \, \bigl| \, |D_t|^{1/2-1/p_c}\bigl(\rho(t)e^{itP_V}f(x)\bigr)\, \bigr|^2 \, dt\, \bigr)^{1/2}
\\
\lesssim \bigl(\, \int_{-\infty}^\infty \sum_{k=0}^\infty |F_k(t,x)|^2 \, dt\, \bigr)^{1/2}
=(2\pi)^{-1/2} \bigl(\int_{-\infty}^\infty \sum_{k=0}^\infty |\Hat F_k(\tau,x)|^2 \, d\tau \, \bigr)^{1/2}.
\end{multline*}
Also, since $p_c>2$, we conclude that this implies that the square of the left side of \eqref{8.4'} is dominated by
$$\sum_{k=0}^\infty\int_{-\infty}^\infty \|\Hat F_k(\tau,x)\|_{L^{p_c}(M)}^2 \, d\tau.$$

Recalling \eqref{8.6}, the support properties of $\Hat \rho$, we see that this along with \eqref{8.5} and orthogonality
imply that the left side of \eqref{8.4'} is dominated by
\begin{align*}
\bigl(\, \sum_{k=0}^\infty \int_{-\infty}^\infty |\tau|^{1-2/p_c} \, &\bigl\|\sum_{\la_j\in [k,k+1)}\Hat \rho(\tau-\la_j)E_jf\bigr\|_{L^{p_c}(M)}^2 \, d\tau\bigr)^{1/2}
\\
&=\bigl(\,  \sum_{k=0}^\infty \int_{k-10}^{k+10} |\tau|^{1-2/p_c} \,
\bigl\| \sum_{\la_j\in [k,k+1)}\Hat \rho(\tau-\la_j)E_jf \bigr\|_{L^{p_c}(M)}^2 \, d\tau \, \bigr)
\\
&\lesssim \, \sum_{k=0}^\infty (1+k)^{1-2/p_c}(1+k)^{2/p_c}\|\chi_k^Vf\|_{L^2(M)}^2
\\
&=\bigl(\, \sum_{k=0}^\infty\|(1+k)^{1/2}\chi_k^Vf\|_{L^2(M)}^2\, \bigr)^{1/2}\approx \|(I+P_V)^{1/2}f\|_{L^2(M)}^2,
\end{align*}
as desired, which completes the proof. \end{proof}

\medskip
\noindent {\bf Remarks.}
We only assumed in Theorem~\ref{strthm} that $H_V\ge0$ to simplify the proof.  Since $H_V$ is bounded from
below due to the assumption that $V\in {\mathcal K}(M)$ this assumption can easily be removed
by applying \eqref{8.2} to the operators where $V(x)$ is replaced by $V(x)+N$ with $N$ sufficiently large.  One
just uses a simple argument involving the Duhamel formula and modifies \eqref{8.2} by replacing the
right side by
$$C_V\bigl(\|(H_V-i)^{1/4}f_0\|_2+\|(H_V-i)^{-1/4}f_1\|_2\bigr).$$

We would also like to remark that this argument shows that Strichartz's estimate \eqref{8.2} can be proven
using the Stein-Tomas restriction theorem \cite{TomasRestriction}.  In \cite{Nicola}, Nicola gives a slightly different proof of this fact.

\section{Analogous results for Schr\"odinger operators in  $\Rn$}\label{Rn}

In this section we shall see that the results that we have obtained for compact manifolds easily extend to the same sort of  results for Schr\"odinger operators in $\Rn$.
In what follows, we shall say that $V\in {\mathcal K}(\Rn)$ if \eqref{a.10} is valid where $B_r(x)$ denotes
the Euclidean ball of radius $r>0$ centered at $x\in \Rn$.  As before, we shall assume that our potentials
are real valued.  Also, we shall let $\Delta$ denote the standard Laplacian on $\Rn$.
If $V\in {\mathcal K}(\Rn)$ it then follows exactly as before that the quadratic form associated with $-\Delta+V(x)$ is defines a unique self-adjoint operator $H_V$ which is
bounded from below.

We can easily modify our arguments for the manifold case to obtain the following analog
of Theorem~\ref{theorem2}:

\begin{theorem}\label{rnthm}  Assume that $V\in
\bigl(L^{\frac{n}2}(\Rn)+L^\infty(\Rn)\bigr)\cap {\mathcal K}(\Rn)$.  Then
if $n=2$ or $n=3$ we have for $\sigma(p)$ as in \eqref{a.3} and $\la\ge 1$ we have
\begin{multline}\label{r.1}
\|u\|_{L^p(\Rn)}\le C_V \la^{\sigma(p)-1}\bigl\| \, \bigl(-\Delta+V-(\la+i)^2\bigr)u\, \bigr\|_{L^2(\Rn)},
\\
\text{if } \, \, 2<p\le \infty \quad \text{and } \, \, u\in \Dom(H_V).
\end{multline}
If $n\ge 4$ this inequality holds for all $2<p<\tfrac{2n}{n-4}$, and we also have for such $n$
\begin{multline}\label{r.2}
\|u\|_{L^p(\Rn)}\le C_V \Bigl( \la^{\sigma(p)-1}\bigl\|  \bigl(-\Delta+V-(\la+i)^2\bigr)u \bigr\|_{L^2(\Rn)}
\\
+\la^{-N+n/2}\bigl\|(I+H_V)^{N/2}R_\la u\bigr\|_{L^2(\Rn)}\Bigr), \\ \text{if } \, \,
p\in [\tfrac{2n}{n-4}, \infty], \, \,  \text{and } \, \, u\in \Dom(H_V), \quad \la\ge1,
\end{multline}
assuming that $N>n/2$ with $R_\la$ being the projection operator for $H_V$ corresponding
to the interval $[2\la^2,\infty)$.
\end{theorem}

Our assumption that
$V\in L^{\frac{n}2}(\Rn)+L^\infty(\Rn)$ means that we can split up $V$ as
$V=V_0+V_1$ where $V_0\in L^{\frac{n}2}(\Rn)$ and $V_1\in L^\infty(\Rn)$.
Before sketching the proof, let us state a couple of corollaries.

As before, if $\chi^V_\la$ is the spectral projection operator for $H_V$ associated with the intervals
$[\la^2,(\la+1)^2]$, then as an immediate corollary of this result we have the following bounds
\begin{equation}\label{r.3}
\|\chi_\la^V f\|_{L^p\Rn)} \le C_V (1+\la)^{\sigma(p)}\|f\|_{L^p(M)}, \quad p\ge 2, \quad \la \ge 0.
\end{equation}
By routine Sobolev estimates for $H_V$, a consequence of Gaussian upper bounds on the heat kernel as in \eqref{s.2} (see e.g. \cite[Theorem B.2.1]{SimonSurvey}), since $H_V$ is bounded from below we also have
\begin{equation}\label{r.4}
\|\chi_{(-\infty,0)}^Vf\|_{L^p(\Rn)}\le C_V\|f\|_{L^2(\Rn)},
\end{equation}
if $\chi_{(-\infty,0)}$ denotes the spectral projection onto the interval $(-\infty,0)$ for $H_V$.

Using \eqref{r.3}--\eqref{r.4} it is straightforward to adapt the proof of Theorem~\ref{strthm} to obtain the following
local Strichartz estimates for $H_V$:
\begin{equation}\label{r.5}
\|u\|_{L^{\frac{2(n+1)}{n-1}}([0,1]\times \Rn)}\le C_V
\bigl(\|(H_V-i)^{1/4}f_0\|_{L^2(\Rn)}+
\|(H_V-i)^{-1/4}f_1\|_{L^2(\Rn)},
\end{equation}
if, $u$ solves the wave equation
\begin{equation}\label{r.6}
(\partial_t^2-\Delta+V(x))u=0, \quad \partial^j_tu|_{t=0}=f_j, \, \, j=0,1.
\end{equation}
We have formulated \eqref{r.5} a bit differently from \eqref{8.2} since we are not assuming here
that $H_V$ is positive.

\medskip

The proof of Theorem~\ref{rnthm} follows from straightforward modifications of the arguments that
we used earlier for the case of compact manifolds.  Let us sketch how one can obtain \eqref{r.1}
when $p=p_c$, $u\in C^\infty_0(\Rn)$ and $n\ge 4$ and leave it up to the reader to verify that the other
cases follow from our earlier arguments.  Note that, as we mentioned before the heat kernel
bounds in \eqref{s.2}, which are due to Aizenman and Simon~\cite{AZ}, are valid here since
we are assuming that $V\in {\mathcal K}(\Rn)$.  Based on this one easily obtains the bounds
for the other exponents $p>2$ when $n\ge4$, and the  arguments that we used to prove
the results for $n=2,3$ in the case of compact manifolds are also straightforward to
adapt to the Euclidean setting.

To prove \eqref{r.1} for $u\in C^\infty_0(\Rn)$, $p=p_c$ and $n\ge 4$, we let
$$T_\la(x,y)=\eta_\delta(x,y)\times (2\pi)^{-n}
\int_{\Rn}\frac{e^{i(x-y)\cdot\xi}}{|\xi|^2-(\la+i)^2} \, d\xi,$$
with, as before $\eta_\delta(x,y)=\eta(|x-y|/\delta)$, where we are fixing
$\eta\in C^\infty_0(\R)$ which equals one on $[-1/2,1/2]$ and is supported
in $(-1,1)$.  We then have the following analog of \eqref{h.2'},
$$(-\Delta-(\la+i)^2)T_\la(x,y)=\delta_y(x)+[\eta_\delta(\, \cdot \, ,y),\Delta] \, T_\la(x,y).$$
Thus, if $-R_\la(x,y)$ equals the last term in the right, we have
$$I=T_\la\circ (-\Delta-(\la+i)^2) +R_\la,$$
if $T_\la$ and $R_\la$ are the integral operators with kernels $T_\la(x,y)$ and
$R_\la(x,y)$, respectively.

Note that these kernels both vanish when $|x-y|>\delta$.  They also are as in \eqref{h.10}--\eqref{h.13} if we
replace $d_g(x,y)$ there by $|x-y|$.  Similarly, we have the analogs of \eqref{h.14}--\eqref{h.16} in our
setting.  Also, since $V\in L^{\frac{n}2}(\Rn)+L^\infty(\Rn)$, we have that
\begin{equation}\label{r.7}
\sup_x \|V\|_{L^{\frac{n}2}(B_\delta(x))}<\e(\delta),
\end{equation}
where $\e(\delta)$ can be made as small as we like by choosing $\delta>0$
small (depending on $V$).

Let us now see how we can use these facts to prove our inequality.  Just like before, we
have $u=T_\la((-\Delta-(\la+i)^2)u)+R_\la u$.  Consequently,
$$u=T_\la\bigl((-\Delta+V-(\la+i)^2)u\bigr)+R_\la u-T_\la(Vu).$$
Since we have the bounds in Proposition~\ref{proph2} and the kernels vanish when $|x-y|>\delta$,
it follows that if $\{Q_j\}$ is a lattice of nonoverlapping cubes in $\Rn$ of sidelength $\delta$, then
\begin{align*}\|T_\la f\|_{L^{p_c}(Q_j)}
&\le C_0 \la^{\sigma(p_c)-1} \|f\|_{L^2(Q_j^*)},
\\
\|T_\la f\|_{L^{p_c}(Q_j)}
&\le C_0  \|f\|_{L^r(Q_j^*)}, \quad \text{if } \, \tfrac1r=\tfrac2n+\tfrac1{p_c},
\end{align*}
and
$$\|R_\la f\|_{L^{p_c}(Q_j)} \le C_\delta \la^{\sigma(p_c)} \|f\|_{L^2(Q^*_j)},$$
where $Q_j^*$ is the cube with the same center as $Q_j$ but four times the side-length.

As a result,
\begin{multline}\label{r.8}
\|u\|_{L^{p_c}(Q_j)}\le C_0 \la^{\sigma(p_c)-1}\|(-\Delta+V-(\la+i)^2)u\|_{L^2(Q^*_j)}
+C_\delta \la^{\sigma(p_c)} \|u\|_{L^2(Q^*_j)}
\\
+C_0\|Vu\|_{L^r(Q^*_j)}.
\end{multline}
 By H\"older's inequality
and \eqref{r.7}
\begin{equation}\label{r.9}
\|Vu\|_{L^r(Q^*_j)}\le \widetilde \e(\delta)\|u\|_{L^{p_c}(Q^*_j)},
\end{equation}
where $\widetilde \e(\delta)$ can be made as small as we like.  Thus, since $\Rn=\bigcup Q_j$ and
the $\{Q^*_j\}$ have finite overlap, if we raise both sides of \eqref{r.8} to the $p_c$-power and
sum both sides of the reulting inequality over $j$, we obtain, similar to before,
$$\|u\|_{L^{p_c}(\Rn)}
\le C\la^{\sigma(p_c)-1}\bigl(
\|(-\Delta+V-(\la+i)^2)u\|_{L^2(\Rn)}+\la\|u\|_{L^2(\Rn)}\bigr) +\tfrac12 \|u\|_{L^{p_c}(\Rn)},
$$
assuming that $\widetilde \e(\delta)$ in \eqref{r.9} is small enough.  This of course yields
\eqref{r.1} as claimed for our $u\in C^\infty_0(\Rn)$.


\subsection{Global results for small potentials}

Let us conclude by showing that we can greatly improve \eqref{r.3} and obtain {\em global} Strichartz
estimates if we assume that $V\in L^{\frac{n}2}(\Rn)$ has small norm and $n\ge 3$.

Before doing this, let us review how we can adapt the arguments from \S \ref{katosection} to see
that, in this case, $-\Delta+V$ is (essentially) self-adjoint and positive.

To see this, we first notice that, by H\"older's inequality,
$$|\langle Vu,u\rangle|\le \|V\|_{L^{\frac{n}2}(\Rn)}\|u^2\|_{L^{\frac{n}{n-2}}(\Rn)}
=\|V\|_{L^{\frac{n}2}(\Rn)} \, \|u\|_{L^{\frac{2n}{n-2}}(\Rn)}^2.$$
By Sobolev's theorem
$$\|u\|_{L^{\frac{2n}{n-2}}(\Rn)}\le C_n\|u\|_{\dot H^1(\Rn)}
=C_n\|\, \sqrt{-\Delta}\, u\|_{L^2(\Rn)}.$$
Thus, if
$a=\|V\|_{L^{\frac{n}2}(\Rn)}C^2_n<1$, i.e.,
\begin{equation}\label{r.10}\|V\|_{L^{\frac{n}2}(\Rn)}<1/C_n^2,
\end{equation}
we have
$$\bigl| \, \langle Vu, u\rangle \, \bigr| \le a\, \langle -\Delta u,u\rangle$$
with $a<1$.    As a result, by the KLMN theorem (see Reed-Simon~\cite[Theorem X.17]{ReedSimon}), the quadratic
form associated with
$-\Delta+V$  defines a unique positive self-adjoint operator $H_V$.

In what follows we shall assume that \eqref{r.10} is valid.  If we make a further assumption
based on the constants in the uniform Sobolev inequalities of Kenig, Ruiz and the third
author~\cite{KRS} we can obtain the following generalization of the Stein-Tomas restriction
theorem~\cite{TomasRestriction}.

\begin{theorem}\label{small}  If  $n\ge4$ there is a $0<\delta_n<1/C_n^2$, where
$C_n$ is as in \eqref{r.10} so that if
  $\|V\|_{L^{\frac{n}2}(\Rn)}<\delta_n$ there is a uniform
constant $C_0=C_0(n)$ so that
\begin{multline}\label{r.11}
\|u\|_{L^{p_c}(\Rn)}\le C_0
\la^{-1+1/p_c}\e^{-1/2}\bigl\|(-\Delta+V-(\la+i\e)^2)u\bigr\|_{L^2(\Rn)},
\\
\text{if } \, \, \e \in (0,\la/2) \, , \,
\text{and } \, u\in \Dom(H_V).
\end{multline}
If $n=3$ this result also holds if, in addition to the assumption that
$\|V\|_{L^{3/2}({\mathbb R}^3)}<\delta_3$, with $\delta_3$ small enough,
we assume that $V\in {\mathcal K}({\mathbb R}^3)$.
\end{theorem}


The reader can check that when $V\equiv0$ \eqref{r.11} is equivalent to the Stein-Tomas restriction
theorem for $\Rn$.  Half of this claim will be used to prove the following special case of
\eqref{r.11} which will be needed for its proof.  Specifically, we shall require the following:

\begin{prop}\label{st}
Fix $n\ge2$.  Then there is a uniform constant $C$ so that if $0<\e<\la/2$
\begin{equation}\label{r.12}
\|u\|_{L^{p_c}(\Rn)}\le C_0
\la^{-1+1/p_c}\e^{-1/2}\bigl\|(-\Delta-(\la+i\e)^2)u\bigr\|_{L^2(\Rn)},
\quad u\in C^\infty_0(\Rn).
\end{equation}
\end{prop}

To prove \eqref{r.12}, we note that, by duality, the inequality is equivalent to the statement that
\begin{align}\label{r.12'}\tag{9.12$'$}
(2\pi)^{-n}\int_{\Rn}\frac{|\Hat f(\xi)|^2}{\bigl|\, |\xi|^2-(\la+i\e)^2\, \bigr|^2} \, d\xi
&=\|(-\Delta-(\la+i\e)^2)^{-1}f\|_{L^2(\Rn)}^2
\\
&\lesssim \la^{-2+2/p_c}\e^{-1}\|f\|_{L^{p_c'}(\Rn)}^2, \quad
\text{if } \, 0<\e<\la/2. \notag
\end{align}

To prove this we shall use the following result which follows from a change of scale and the
Stein-Tomas~\cite{TomasRestriction}  $L^2$-restriction theorem for the Fourier transform (the $r=1$ case):
\begin{equation}\label{r.13}
\Bigl(\,  \int_{S^{n-1}} |\Hat f(r\omega)|^2 r^{n-1}\, d\omega\, \Bigr)^{1/2}
\le C_0\,  r^{1/p_c} \, \|f\|_{L^{p_c'}(\Rn)}.
\end{equation}

\begin{proof}[Proof of Proposition~\ref{st}]  As we just noted, it suffices to prove \eqref{r.12'}.  By
\eqref{r.13}, the left side of this inequality is majorized by
$$\int_0^\infty \frac{r^{2/p_c}\, dr}{\bigl| \, r^2-(\la+i\e)^2\, \bigr|^2}\ \, \times \, \|f\|^2_{p_c'},$$
which yields \eqref{r.12'} as,
\begin{multline*}
\int_0^\infty \frac{r^{2/p_c}\, dr}{\bigl| \, r^2-(\la+i\e)^2 \, \bigr|^2}
= \la \, \cdot\,  \la^{-4+2/p_c}\int_0^\infty \frac{r^{2/p_c}\, dr}{\bigl| \, r^2-(1+i\e/\la)^2\, \bigr|^2}
\\
\approx \la^{-3+2/p_c}\, \cdot \, (\la/\e) = \la^{-2+2/p_c} \e^{-1},
\end{multline*}
as desired, since we are assuming that $0<\e<\la/2$.
\end{proof}

\begin{proof}[Proof of Theorem~\ref{small}]
Let us first handle the case where $n\ge4$.  Since $C^\infty_0(\Rn)$ then is an operator core for $-\Delta+V$ (cf. \cite[Theorem B.1.6]{SimonSurvey}),
it suffices to prove the inequality for such $u$.  To do so, we write
$$u=\bigl(-\Delta-(\la+i\e)^2\bigr)^{-1} \bigl(-\Delta+V-(\la+i\e)^2)u
-\bigl(-\Delta-(\la+i\e)^2\bigr)^{-1}(Vu)\, =\, I+II.$$
By Proposition~\ref{st}, the $L^{p_c}(\Rn)$ norm of  $I$ is dominated by
the right side of \eqref{r.11}.  By the uniform Sobolev estimate \eqref{a.7} from \cite{KRS},
the $L^{p_c}(\Rn)$ norm of $II$ is dominated
by
$$A_n\|Vu\|_{L^r(\Rn)}, \quad \text{where } \, \, 1/r=2/n+1/p_c,$$
and $A_n$ is a uniform constant.
Hence, by H\"older's inequality,
$$\|II\|_{L^{p_c}(\Rn)}\le A_n\|V\|_{L^{\frac{n}2}(\Rn)}\|u\|_{L^{p_c}(\Rn)}.$$
which, together with the bound for $I$ yields \eqref{r.11} if $\delta_n<(2A_n)^{-1}$.

It is straightforward to see that the arguments in \S \ref{3dsection} can be used
to show that \eqref{r.12} and \eqref{a.7} imply the 3-dimensional result.  One repeats
the duality argument, noting that the step that involves Fubini's theorem is justified due to the
fact that if $K_{\la,\e}$ denotes the kernel of $\bigl(-\Delta-(\la+i\e)^2\bigr)^{-1} $
then
$$\sup_x \int_{\R^3}|K_{\la,\e}(x-y)| \, |V(y)| \, |u(y)|\, dy<\infty,$$
due to the fact that if $u\in \Dom(-\Delta+V)$ then $u\in L^p(\R^3)$ for all $2<p\le \infty$
\footnote{This follows from \cite[Theorem B.2.1]{SimonSurvey}.}
and $|K_{\la,\e}(x)|\lesssim |x-y|^{-1}$.  One uses the assumption that $V\in {\mathcal K}(\R^3)$
to control the integral in the region where $|x-y|\le 1$.  One then controls the remaining part
using H\"older's inequality since $u(\, \cdot\, )K_{\la,\e}(x-\, \cdot \, )\in L^3(|y|>1)$
 and we are also assuming that $V\in L^{3/2}(\Rn)$.
\end{proof}

Next, just as before, we can use the spectral theorem to show that the quasimode estimates \eqref{r.11} yield
related (and indeed equivalent by the arguments in \cite{SoggeZelditchQMNote}) bounds for spectral
projection operators:

\begin{corollary}\label{cor9.4}  Let $n\ge3$ and let $V$ be as in Theorem~\ref{small}.
If $P_V=\sqrt{H_V}$ let $\chi_{[\la,\la+\e)}^V$ denote the associated spectral projection operators
corresponding to the interval $[\la,\la+\e)$.  Then
\begin{equation}\label{r.14}
\|\chi_{[\la,\la+\e)}^V f\|_{L^{p_c}(\Rn)}\le C_V \la^{1/p_c}\e^{1/2} \, \|f\|_{L^2(\Rn)}, \quad
\text{if } \, \, \la \ge 2\e,
\end{equation}
and
\begin{equation}\label{r.15}
\|\chi^V_{[0,2\e)}f\|_{L^{p_c}(\Rn)} \le C_V \e^{\frac12+\frac1{p_c}} \|f\|_{L^2(\Rn)}.
\end{equation}
\end{corollary}

\begin{proof}
To prove \eqref{r.14}, we note that if $\la\ge 2\e$ and $\tau\in [\la,\la+\e)$ then
$$|\tau^2-(\la+i\e)^2|\le C_0\la \e.$$
Consequently, by the spectral theorem
$$\|\chi^V_{[\la,\la+\e)}f\|_{L^2(\Rn)}\le C_0\la \e\|f\|_{L^2(\Rn)}.$$
If we use this and the quasimode estimates \eqref{r.12}, we obtain
\eqref{r.14}.

To prove \eqref{r.15} we take $\la=2\e$ in \eqref{r.12}:
\begin{equation}\label{r.16}\|u\|_{L^{p_c}(\Rn)}\le C_V \e^{-\frac32+\frac1{p_c}}
\|(-\Delta+V-(2\e+i\e)^2)u\|_{L^2(\Rn)}
\end{equation}
Since for $0\le \tau \le 2\e$, $|\tau^2-(2\e+i\e)^2|\approx \e^2$, by the spectral theorem
$$\|(-\Delta+V-(2\e+i\e)^2)\chi^V_{[0,2\e)}f\|_{L^2(\Rn)} \lesssim \e^2 \|f\|_{L^2(\Rn)}.$$
Hence, by \eqref{r.16} we have \eqref{r.15}.
\end{proof}

\medskip
\noindent{\bf Remarks.}  It was noted in Ionescu and Jerison~\cite{IonescuJerison} that if $\|V\|_{L^{n/2}(\Rn)}$ is
small enough then $-\Delta+V$ cannot have eigenvalues.  The argument at the beginning of this section shows
that, under this assumption, the operator can have no negative spectrum, and \eqref{r.14}--\eqref{r.15} imply
that there can be no eigenvalues in $[0, \infty)$.

There has been much work in recent years in trying to obtain bounds of the form \eqref{r.12} or \eqref{r.14} on
compact Riemannian manifolds $(M,g)$
when $\e=\e(\la)$ is a function of $\la$ and where $p_c$ may be replaced by other exponents.  See, e.g.,
 \cite{Berard}, \cite{BSSY},  \cite{HassellTacyNonPos}, \cite{BlairSoggeToponogov} and\cite{BlairSoggelog}.

 It would be interesting to see whether one could replace the smallness condition in Theorem~\ref{small} by ones that
 are analogous to those in \cite{IonescuJerison} or \cite{iRodS}.  The reader can check that if the global Kato norm, as defined in Rodnianski and Schlag~\cite{iRodS}, is smaller than $4\pi$ when $n=3$, then one has the variant of
 \eqref{r.12} corresponding to $p=\infty$ (where $\la^{-1+1/p_c}$ is replaced by $1$ in the right side).
 Rodnianski and Schlag showed that under this smallness assumption one has the natural dispersive
 estimates for $e^{itH_V}$, and they also improved on the related earlier results of Journ\'e, Soffer and the
 third author~\cite{JSS} in terms of assumptions on the potentials $V(x)$ that are needed for such dispersive estimates.  It would be interesting to see whether such hypotheses could lead to bounds of the
 form \eqref{r.12}.

 \medskip

 Let us conclude by presenting another estimate which breaks down if there are embedded eigenvalues: Global
 Strichartz estimates.

 \begin{theorem}\label{rnstr}  Let $n\ge3$ and $V$ be as in Theorem~\ref{small}.  Then, if $P_V=\sqrt{H_V}$,
 \begin{equation}\label{r.17}
 \|u\|_{L^{p_c}(\R\times \Rn)}\le C_V\|P_V^{1/2}f_0\|_{L^2(\Rn)}
 +\|P_V^{-1/2}f_1\|_{L^2(\Rn)},
 \end{equation}
 if $u$ solves the wave equation associated to $H_V$ with initial data $(f_0,f_1)$, i.e.,
 $$(\partial^2_t-\Delta+V)u=0, \quad
 \partial_t^ju|_{t=0}=f_j, \, \, j=0,1.$$
 \end{theorem}

 When $n=3$, Bui, Duong and Hong~\cite{Bui} obtained results of this type
 (as well as the stronger dispersive estimates) under an assumption that requires
 a global Kato norm of $V$ to be finite.

 \begin{proof}
 It suffices to see that there is a uniform constant $C_0(V,n)$ so that for $0<\e<1$
 \begin{equation}\label{r.17'}\tag{9.17$'$}
 \|e^{itP_V}f\|_{L^{p_c}([-\e,\e]\times \Rn)}
 \le C_0\|(P_V+\e I)^{1/2}f\|_{L^2(\Rn)}.
 \end{equation}

 To prove \eqref{r.17'}, similar to the proof of Theorem~\ref{strthm}, it suffices to show that
 if we fix $\rho\in {\mathcal S}(\R)$ with $\text{supp }\Hat \rho\subset (-1/2,1/2)$, then we have
 the uniform bounds
 \begin{equation}\label{r.18}
 \|\rho(\e t)\, e^{itP_V}f\|_{L^{p_c}(\R\times \Rn)} \le C_0\|(P_V+\e I)^{1/2}f\|_{L^2(\Rn)}.
 \end{equation}

 In order to verify this, let
 $$I_k=[(k-1)\e, k\e), \quad k=1,2,3,\dots.$$
 Then if $\chi_k$ is the spectral projection operator for $I_k$ associated with $P_V$, it follows
 from Corollary~\ref{cor9.4} that
 \begin{equation}\label{r.19}
 \|\chi_kf\|_{L^{p_c}(\Rn)}\lesssim \e^{1/2} \, (k\e)^{1/p_c} \, \|\chi_kf\|_{L^2(\Rn)},
 \quad k=1,2,3,\dots.
 \end{equation}
 Also,
 $$\sum_{k=1}^\infty \|\chi_kf\|_{L^2(\Rn)}^2=\|f\|_{L^2(\Rn)}^2.$$
 As before, we use $L^2_t\to L^{p_c}_t$ Sobolev estimates to deduce that
 $$\|\rho(\e t)\, e^{itP_V}f\|_{L^{p_c}_{t,x}}\lesssim
 \| \, |D_t|^{1/2-1/p_c} (\rho(\e t)e^{itP_V}f)\, \|_{L^{p_c}_xL^2_t}.
 $$

 Let
 $$F(t,x)=|D_t|^{1/2-1/p_c}\bigl(\rho(\e t)e^{itP_V}f\bigr).$$
 If we take the Fourier transform in $t$, we deduce that
 $$\Hat F(\tau,x)=|\tau|^{1/2-1/p_c} \, \e^{-1} \, \bigl(\Hat \rho(\e^{-1}(\tau-P_V)) f\bigr)(x)
 =\sum_{k=1}^\infty \Hat F_k(\tau,x),$$
 where
 $$\Hat F_k(\tau,x)=|\tau|^{1/2-1/p_c} \, \e^{-1} \, \bigl(\Hat \rho(\e^{-1}(\tau-P_V))\circ\chi_k f\bigr)(x).$$

 Note that since $\text{supp }\Hat \rho\subset (-1/2,1/2)$,
 $\Hat \rho(\e^{-1}(\tau-P_V))\circ\chi_k$ if $\tau\notin [\e(k-10),\e(k+10)]$.  Consequently,
 $$\int_{-\infty}^\infty
 F_k(t,x) \, \overline{F_\ell(t,x)} \, dt
 =(2\pi)^{-1}\int_{-\infty}^\infty \Hat F_k(\tau,x)\, \overline{\Hat F_\ell(\tau,x)} \, d\tau
 =0 \quad \text{if } \, \, |k-\ell|>10.$$
 As a result,
 \begin{align*}\bigl(\int_{-\infty}^\infty |F(t,x)|^2 \, dt\bigr)^{1/2}
 =\bigl(\int_{-\infty}^\infty \bigl| \, \sum_{k=1}^\infty F_k(t,x)\, \bigr|^2 \, dt\bigr)^{1/2}
& \lesssim \bigl(\sum_{k=1}^\infty \int_{-\infty}^\infty |F_k(t,x)|^2 \, dt\bigr)^{1/2}
\\
&\le  \bigl(\sum_{k=1}^\infty \int_{-\infty}^\infty |\Hat F_k(\tau,x)|^2 \, d\tau\bigr)^{1/2}.
 \end{align*}

 By combining this with \eqref{r.19} and the above we deduce that
 \begin{align*}
  \|\rho(\e t)\, e^{itP_V}f\|_{L^{p_c}(\R\times \Rn)}^2&\lesssim \sum_{k=1}^\infty \int_{-\infty}^\infty
  \| \Hat F_k(\tau,\, \cdot \, )\|^2_{L^{p_c}(\Rn)} \,
  d\tau
  \\
  & =\e^{-2} \sum_{k=1}^\infty \int_{(k-10)\e}^{(k+10)\e} |\tau|^{1-2/p_c}
  \|\Hat \rho(\e^{-1}(\tau-P_V))\chi_kf\|_{p_c}^2 \, d\tau
  \\
  &\lesssim \e^{-2}\sum_{k=1}^\infty \e\cdot (k\e)^{1-2/p_c}\|\chi_k(\rho(\e^{-1}(\tau-P_V))f) \,    \|_{p_c}^2
  \\
  &\lesssim \e^{-2}\sum_{k=1}^\infty \e \cdot (k\e)^{1-2/p_c}\, \bigl(\e^{1/2}(\e k)^{1/p_c}\bigr)^2
  \|\chi_kf\|_2^2
  \\
  &=\e^{-2}\sum_{k=1}^\infty \e^2 \, (k\e)\, \|\chi_kf\|_2^2
  \\
  &=
  \sum_{k=1}^\infty \bigl\| \, (k\e)^{1/2}\chi_kf \, \bigr\|_2^2 \approx
  \bigl\| \, (P_V+\e I)^{1/2}f\, \bigr\|_2^2,
 \end{align*}
 as desired.
 \end{proof}


\section*{Acknowledgements}
The authors are grateful to W. Schlag and K.-T. Sturm for helpful suggestions and comments.  The research was also carried out in part while the third author was visiting the Mittag-Leffler Institute, and he wishes to thank the institute for its hospitality and the feedback received from fellow visitors, especially R. Killip.  The research was also carried in part while this author was visiting the University of Edinburgh and the University of Birmingham and he also wishes to thank these institutions for their hospitality.

\bibliography{refs}
\bibliographystyle{amsplain}

\end{document}